\documentclass[english,a4paper,reqno,11pt]{amsart} 
\usepackage[utf8]{inputenc} 
\usepackage[T1]{fontenc} 
\usepackage{lmodern} 
\usepackage{pdfrender}
\usepackage{amsmath,bm}
\usepackage{amssymb} 
\usepackage{stmaryrd} 
\usepackage{mathtools} 
\usepackage{mathrsfs}  
\usepackage{graphicx} 
\usepackage{xcolor}   
\usepackage[english]{babel} 
\usepackage{amsthm, thmtools} 
\usepackage{caption} 
\usepackage{subcaption}
\usepackage{varioref} 
\usepackage{microtype} 
\usepackage[shortlabels]{enumitem} 
\usepackage{xargs} 
\usepackage[colorinlistoftodos,
	prependcaption,
	textsize=scriptsize
	]
	{todonotes} 
\usepackage{csquotes} 
\usepackage{dsfont} 
\usepackage{bbm} 
\usepackage{tikz,tikz-cd} 
\tikzcdset{scale cd/.style={every label/.append style={scale=#1},
    cells={nodes={scale=#1}}}}
\usetikzlibrary{positioning}
\usepackage{shuffle} 
\usepackage{blkarray} 
\setlist[itemize]{
  nosep,
  align=left,
  leftmargin=0pt,
  labelwidth=1.25em,
  itemindent=1.25em,
  labelsep=0pt,
}
\usepackage[
	doi=false,
	isbn=false,
	backend=biber,
	style=alphabetic,
	maxbibnames=99,
	maxalphanames=99,
	giveninits=true,
	]
	{biblatex} 
\usepackage[hypertexnames=false]{hyperref} 

\newcommandx{\lz}[2][1=]{\todo[inline,linecolor=blue,backgroundcolor=blue!25,bordercolor=blue,#1, author = LORENZO]{#2}} 
\newcommandx{\jdj}[2][1=]{\todo[inline,linecolor=red,backgroundcolor=red!25,bordercolor=red,#1, author = JEAN-DAVID]{#2}} 

\newcommand*{\fatten}[1][.4pt]{%
  \textpdfrender{
    TextRenderingMode=FillStroke,
    LineWidth={\dimexpr(#1)\relax},
  }%
}

\setenumerate[0]{label = \textit{\arabic*.}, ref = \textit{\arabic*.}} 

\numberwithin{equation}{section} 

\newtheorem{prop}{Proposition}[section]
\newtheorem{definition}[prop]{Definition}
\newtheorem{definition-theorem}[prop]{Definition-Theorem}
\newtheorem{lemma}[prop]{Lemma}

\newtheorem{theorem}[prop]{Theorem}
\newtheorem{remark}[prop]{Remark}
\newtheorem{corollary}[prop]{Corollary}

\newcommand{\0}{{\bf 0}}
\newcommand{\1}{\mathbf{1}}
\newcommand{\m}{{\bf m}}

\newcommand{\n}{{\bf n}}

\newcommand{\p}{{\bf p}}
\newcommand{\q}{{\bf q}}
\newcommand{\e}{{\bf e}}

\newcommand{\z}{\mathsf{z}}
\newcommand{\A}{\mathcal{A}}
\newcommand{\Abar}{\overline\A}
\newcommand{\B}{\mathcal{B}}
\newcommand{\Bbar}{\overline{\mathcal{B}}}

\newcommand{\D}{\mathcal{D}}

\newcommand{\N}{\mathbb{N}}

\newcommand{\R}{\mathbb{R}}
\newcommand{\sym}{\mathsf{Sym}}
\newcommand{\tens}{\mathsf{Tens}}

\newcommand{\tors}[1]{\mathbf{T}^{#1}}
\newcommand{\curv}[1]{\mathbf{R}^{#1}}

\newcommand{\env}{\mathcal{U}}
\newcommand{\envU}{\mathcal{U}_{[\cdot,\cdot]}}

\newcommand{\X}{\mathfrak{X}}
\newcommand{\Z}{\mathbb{Z}}

\newcommand{\lb}{\fatten[0.5pt]{[}}
\newcommand{\rb}{\fatten[0.5pt]{]}}
\newcommand{\lp}{\left(}
\newcommand{\rp}{\right)}
\newcommand{\rhohat}{\hat{\rho}}
\newcommand{\otimesbold}{\bm{\otimes}}

\newcommand{\id}{\textnormal{id}}
\newcommand{\End}{\textnormal{End}}
\newcommand{\Der}{\textnormal{Der}}
\newcommand{\Hom}{\textnormal{Hom}}
\newcommand{\conc}{\mathsf{conc}}

\newcommand{\ind}{\mathds{1}}

\newcommand{\tr}{\triangleright}
\newcommand{\td}{\scalebox{1.2}[0.7]{$\Diamond$}}
\newcommand{\la}{\left\langle}
\newcommand{\ra}{\right\rangle}

\newcommand{\writefun}[5]{\ensuremath{\begin{array}[t]{lrcl}
#1 : & #2 & \longrightarrow & #3 \\
    & #4 & \longmapsto & #5 \end{array}}} 

\newcommand{\writefuniso}[5]{\ensuremath{\begin{array}[t]{lrcl}
#1 : & #2 & \overset{\sim}{\longrightarrow} & #3 \\
    & #4 & \longmapsto & #5 \end{array}}} 

\newcommand{\C}{\mathcal{C}}
\newcommand{\ass}{\mathsf{a}}

\newcommand{\M}{\mathcal{M}}

\newcommand{\rhobar}{\overline{\rho}}

\makeatletter
\newcommand{\Scalecenter}[3]{#1{\mathpalette\Scalecenter@{{#2}{#3}}}}
\newcommand{\Scalecenter@}[2]{\Scalecenter@@#1#2}
\newcommand{\Scalecenter@@}[3]{%
  \vcenter{\hbox{\scalebox{#2}{$\m@th#1#3$}}}%
}
\newcommand{\btr}{%
  \Scalecenter{\mathord}{0.75}{\,\blacktriangleright\,}%
}
\newcommand{\ssum}{%
  \DOTSB\Scalecenter{\mathop}{0.75}{\sum}\slimits@
}

\makeatother

\title[Geometric post-Lie deformations of post-Lie algebras.]{Geometric post-Lie deformations of post-Lie algebras and regularity structures.}
\author{Jean-David JACQUES}
\address{Institut für Mathematik, Universität Potsdam, Germany}
\email{jean-david.jacques@uni-potsdam.de}
\date{\today}

\begin{document}

\begin{abstract} In order to derive a class of geometric-type deformations of post-Lie algebras, we first extend the geometrical notions of torsion and curvature for a general bilinear operation on a Lie algebra, then we derive compatibility conditions which will ensure that the post-Lie structure remains preserved.\\
This type of deformation applies in particular to the post-Lie algebra introduced in \cite{jacques2023post} in the context of regularity structures theory. We use this deformation to derive a pre-Lie structure for the regularity structures approach given in \cite{LOT}, which is isomorphic to the post-Lie algebra studied in \cite{jacques2023post} at the level of their associated Hopf algebras.\\
In the case of sections of smooth vector bundles of a finite-dimensional manifold, this deformed structure contains also, as a subalgebra, the post-Lie algebra structure introduced in \cite{munthe2013post} in the geometrical context of moving frames.
\end{abstract}

\maketitle

\tableofcontents

\section{Introduction}
Recent developments in the rough path theory and later more broadly in the theory of regularity structures \cite{Hai14} has led to an intensive study of some preexisting old algebraic structures, such as pre-Lie algebras, and to explain their relations to graded Hopf algebras as the Connes-Kreimer algebra or the Grossmann-Larson algebra, but also to develop new ones suitable for the purpose of renormalization of SPDEs.\\
As in the seminal article on the article of regularity structures theory \cite{Hai14}, the Picard iteration method for expressing solutions to SPDEs involves constructing tree-like nested iterated integrals against convolution kernels. , which renders a B-series-like sums over trees (see \cite{Butcher} for B-series). Operations on B-series, such as substitution, can be described using a coproduct on the free vector space generated by rooted trees \cite{CMEF}—namely, the renowned Butcher–Connes–Kreimer coproduct \cite{connes1999hopf}. The dual of this coproduct, known as the Grossman–Larson product \cite{grossmanlarson}, is closely related to the \textit{grafting pre-Lie algebra of rooted trees} \cite{Hoffman} which turns out to be the free pre-Lie algebra \cite{CL}. In the context of regularity structures, however, the underlying algebraic framework is richer, as the trees involved carry various types of decorations. This has lead to a multi-pre-Lie algebra structure \cite{foissy2021algebraic} \cite{bruned2019algebraic} \cite{bruned2023algebraic}.\\

We recall first some basic notations and definitions:
\begin{definition}
    Given a bilinear operation $\ast$ on a vector space $L$, we denote
    \begin{itemize}
    \item $[\cdot,\cdot]_\ast$ the \textbf{commutator} of $\ast$, defined as:
    \[[x,y]_\ast:=x\ast y-y\ast x, \qquad  x,y\in L\]
    \item $\ass_\ast$ the \textbf{associator} of $\ast$, defined as:
    \[\ass_\ast(x,y,z):=x\ast (y\ast z)-(x\ast y)\ast z,  \qquad x,y,z\in L\]
    \end{itemize} 
\end{definition}

\begin{definition}\label{def: pre-Lie}
A (left) \textbf{pre-Lie algebra} $(L,\tr,)$ is a vector space $L$ endowed with a bilinear operation $\tr:L\otimes L\rightarrow L$ which satisfies for all $x,y,z\in L$ the following equality:
\begin{equation}\label{eq: pre-Lie equality}
    \ass_\tr(x,y,z)=\ass_\tr(y,x,z)
\end{equation}
\end{definition}
This last identity \eqref{eq: pre-Lie equality} ensures that the commutator $[\cdot,\cdot]_\tr$ satisfies the Jacobi identity, and hence $(L,[\cdot,\cdot]_\tr)$ is a Lie algebra.

\begin{definition}\label{def: post-Lie}
A (left) \textbf{post-Lie algebra} $(L,\tr,[\cdot,\cdot])$ is a vector space $L$ endowed with two bilinear operations $\tr,[\cdot,\cdot]:L\otimes L\rightarrow L$ which satisfy for all $x,y,z\in L$ the following conditions:
\begin{enumerate}
    \item $[\cdot,\cdot]$ is a Lie bracket
    \item $\tr$ is a derivation on $(L,[\cdot,\cdot])$, that is to say:
     \begin{equation}\label{eq: post-Lie derivation condition}
        x\tr[y,z]=[x\tr y,z]+[y,x\tr z]
    \end{equation}
    \item \begin{equation}\label{eq: post-Lie associator condition}
        [x,y]\tr z=\ass_\tr(x,y,z)-\ass_\tr(y,x,z)
    \end{equation}
\end{enumerate}
\end{definition}
Post-Lie algebras form a category that includes both Lie and pre-Lie algebras as subcategories. Indeed:
\begin{itemize}
    \item a Lie algebra $(L,[\cdot,\cdot])$ can be seen as a post-Lie algebra $(L,\tr,[\cdot,\cdot])$ with null post-Lie product $\tr\equiv 0$,
    \item a pre-Lie algebra $(L,\tr)$ can be seen as a post-Lie algebra $(L,\tr,[\cdot,\cdot])$ with null Lie bracket $[\cdot,\cdot]\equiv 0$.
\end{itemize}~

Originally, the notion of pre-Lie algebras appeared simultaneously in two different approaches: in Gerstenhaber \cite{gerstenhaber1963cohomology} for Hochschild cohomology and in Vinberg \cite{vinberg1963} for differential geometry to handle tangential vector fields. Later, the more general notion of post-Lie algebra was first discovered by B.~Vallette \cite[\S 4.3.3]{vallette2007homology} in the context of
purely operadic questions related to Koszul dualization of the commutative trialgebra operad. The free post-Lie algebra has been first described in \cite{munthe2003enumeration} (before the formalisation of post-Lie algebras) and then later given as the free Lie algebra over the free magma over a set of indices in \cite{munthe2013post}, in which paper, the authors also observed that in the case of classical differential geometry, some hypothesis on the connection on a smooth finite dimensional manifold turn the connection along with the torsion into a post-Lie structure on the tangent bundle of the manifold.
We refer to the books \cite{lee2018introduction} and \cite{nakahara2018geometry} for the theory of smooth differential manifolds and to \cite{curry2017post} for relations between post-Lie algebras and Lie group integration.

Recently, the notion of post-Lie algebra has got a significant impact on the theory of regularity structures and especially concerning the tree-free approach of \cite{LOT}, in which the authors studied a class of semi-linear (S)PDEs on $\R^2$ of the type:
\begin{equation}\label{eq: semi-linear}
    \mathcal{L}(u)=a(u)\,\xi,
\end{equation}
where, $\xi$ is a rough driver (typically the gaussian white noise), $\mathcal{L}$ is a linear differential operator, and $a$ is a real function called the \textit{non-linearity}. The main tool for building the structure group are \textit{multiindices}, which are compactly supported functions $\beta:\N\times \N^2\raisebox{0.2pt}{\footnotesize{$\setminus\{0_2\}$}} \to \N$. In that context, the free polynomial algebra $\A=\R[\{\z^\beta\}_\beta]$, with multiplication $\z^{\beta}\cdot\z^{\beta'}=\z^{\beta+\beta'}$, is used as base space, along with a set of derivations on $\A$, which are of two types: the \textit{tilt derivations} $\{\partial_i\}_{i\in\{1,2\}}$, and the \textit{shift derivations} $\{D^{\n}\}_{\n\in \N^2}$.

In \cite{bruned2022post} post-Lie relations in the context of \cite{LOT} were derived and in \cite{jacques2023post} a strong algebraic framework has been built allowing for an extensive study. The algebraic framework developed in \cite{jacques2023post} is based on a post-Lie algebra structure on $\A\otimes \Der(\A)$ (which remain valid for every commutative and associative algebra $\A$). Its canonical nature doesn't require any extra hypothesis and it avoids the need to define an $\A$-free module structure on a subspace of $\Der(\A)$. The bracket of that post-Lie algebra should encode the non-commutativity of the derivations which is the case for regularity structures, the tilt and shift derivations don't commutes with each other.\\

The goal of that paper is twofold:
\begin{itemize}[align=parleft]
    \item Defining a class of deformation of post-Lie algebras, called \textbf{geometric post-Lie deformations}, abbreviated \textbf{gpL-deformations}, by defining for post-Lie structures $(\tr,[\cdot,\cdot])$ a linear perturbation by a bilinear operation $\td$, aiming to a new structure $(\tr+\td,[\cdot,\cdot]-[\cdot,\cdot]_{\diamond})$. Certain compatibility conditions between $\td$ and $(\tr,[\cdot,\cdot])$, will ensure that $(L,\tr+\td,[\cdot,\cdot]-[\cdot,\cdot]_{\diamond})$ forms again a post-Lie algebra. An interesting particularity is that the original post-Lie algebra and its gpL-deformation have their associated envelopping Hopf algebras, which are isomorphic.\\
    GpL-deformations of the post-Lie algebra of \cite{jacques2023post} give a whole class of post-Lie algebra structures on $\A\otimes\Der(\A)$. In particular, it permits to establish a close relation between the post-Lie structure in \cite{jacques2023post} and the one in \cite{munthe2013post} in the case where $L=\mathfrak X(\M)$ is the space of smooth vector bundle on a (finite dimensional) manifold $\M$.\\
    Note that recently in \cite{bruned2025post} (and independently of this present work), the authors were interested in similar deformations of pre-Lie algebras $(L,\tr)$ in which the pre-Lie product is perturbed additively by a bilinear operation $\omega$, leading to a post-Lie algebra $(L,\tr+\omega,\pi)$, where $\pi$ denotes a Lie bracket on $L$, the couple $(\omega,\pi)$ being characterised as a Maurer-Cartan element of a certain differential graded Lie algebra, see \cite[Theorem 2.5]{bruned2025post}.
    \item Constructing a pre-Lie algebra structure on the regularity structures space $L\subset \R[\{\z_k,\z_\n\}]\otimes \Der(\R[\{\z_k,\z_\n\}])$, as a $\td$-gpL-deformation of the post-Lie algebra of \cite{jacques2023post}, where $\td$ should be defined and studying its \textit{structure group} $G$, along with the recentering maps: $\{\Gamma_f^{\btr}\}_{f\in G}$.
\end{itemize}~\

The paper is organized as follows:\\
In section \ref{sec: torsion and curvature}, for sake of generality, we adopt a broad point of view starting with a general Lie algebra $(L,[\cdot,\cdot])$ endowed with a bilinear operation $\td$. In that context, we give an algebraic definition of the notion of \textit{torsion} and \textit{curvature} of $\td$ on $(L,[\cdot,\cdot])$.
A proof of the first Bianchi identity is provided in Appendix Section \ref{subsec: Appendix proofs}, in the most general setting of a connection that may have non-vanishing torsion. This is included for the sake of completeness, as references in the current literature are rather scarce and difficult to locate, and since this identity is the key to derive conditions for post-Lie structure in Theorem \ref{th: geometric post-Lie algebra}. Then we study gpL-deformation of post-Lie algebras in Theorem \ref{th: deformation post-Lie} by deriving compatibility conditions. We have relegated to the Appendix Subsection \ref{subsec: PL deformation in coordinates} the study of gpL deformation in coordinates for a fixed basis, where in Proposition \ref{prop: conditions coordinates} we give the polynomial equations characterising all possible gpL deformations.\\
In section \ref{sec: deformation post-Lie alg of derivations}, we apply the results of the preceding section to the the post-Lie algebra of derivations of \cite{jacques2023post} in Theorem \ref{theo: post-Lie structure from derivations}. This amounts to define the class of deformations in that case. Note that in the Appendix subsection \ref{subsec: appendix M-K--L post-Lie alg}, given a finite dimensional smooth manifold, we specialize our results for the algebra $\A=\C^\infty(\M,\R)$ endowed with the pointwise product, given a connection $\nabla$ on the smooth sections of the tangent bundle $\X(\M)$, which is flat and has constant torsion. In that context, we can identify the Munthe-Kaas-Lundervold (M-K--L) post-Lie algebra $(\mathfrak X(\M),\nabla,-\tors{\nabla,[\cdot,\cdot]_J})$ as a sub post-Lie algebra of $(\mathcal C^\infty(\M,\R)\otimes \mathfrak X(\M),\btr,\lb\cdot,\cdot\rb)$, with $\btr:=\tr+\nabla$ which is the $\nabla$-gpL-deformation of the post-Lie algebra $(\A\otimes\Der(\A),\tr,[\cdot,\cdot])$ of \cite{jacques2023post}.\\
Finally, in section \ref{sec: regularity structures} we apply the a gpL-deformation to the post-Lie algebra structure $L\subset \A\otimes \Der(\A)$ studied in \cite{jacques2023post} and given here in Theorem \ref{theo: post-Lie algebra structure for L} where $L$ is given by \eqref{eq: sub-post-Lie algebra L} in order to derive a pre-Lie algebra structure $(L,\tr+\td)$ that is adapted for regularity structures, in the sense that key finiteness conditions of Propositions \ref{prop: finiteness btr} and \ref{prop: finiteness representation} are satifyed, allowing to define the structure group and recentering maps.\\

The following diagram elucidates the logical structure of this paper (in plain arrows), where the only dashed arrow indicates a link explained in Subsection \ref{subsec: appendix M-K--L post-Lie alg}

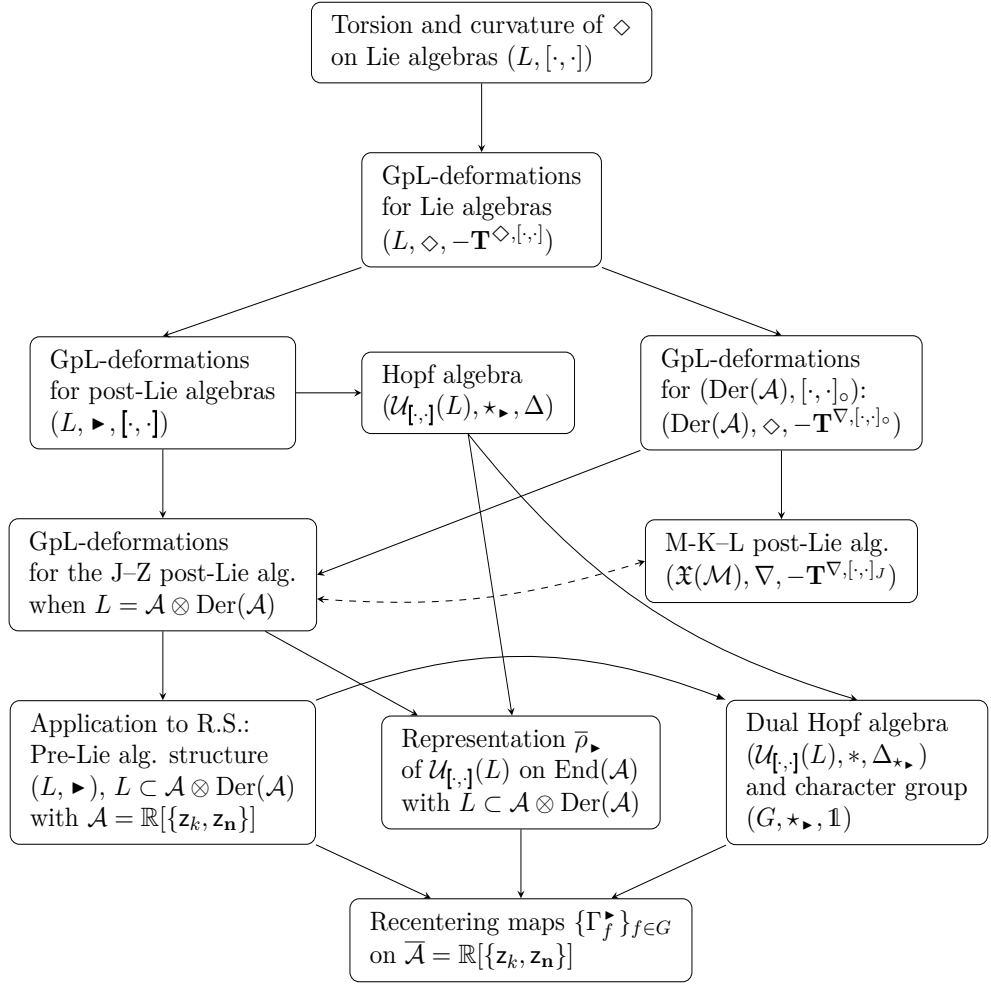
\resizebox{13 cm}{13 cm}{
\begin{tikzpicture}[>=stealth,every node/.style={shape=rectangle,draw,rounded corners},]
    \node (c0) {\begin{tabular}{l} Torsion and curvature of $\td$\\ on Lie algebras $(L,[\cdot,\cdot])$\end{tabular}};
    \node (c1) [below =of c0]{\begin{tabular}{l} GpL-deformations\\ for Lie algebras \\ $(L,\td,-\tors{\td,[\cdot,\cdot]})$\end{tabular}};
    \node (c2) [below left=of c1]{\begin{tabular}{l} GpL-deformations\\ for post-Lie algebras \\$(L,\btr,\lb\cdot,\cdot\rb)$\end{tabular}};
    \node (c2') [right=of c2]{\begin{tabular}{l} Hopf algebra \\$(\env_{\lb\cdot,\cdot\rb}(L),\star_{\btr},\Delta)$\end{tabular}};
    \node (c3) [below =of c2]{\begin{tabular}{l} GpL-deformations\\ for the J--Z post-Lie alg. \\when $L=\A\otimes \Der(\A)$\end{tabular}};
    \node (c4) [below =of c3]{\begin{tabular}{l}Application to R.S.:\\ Pre-Lie alg. structure\\ $(L,\btr)$, $L\subset\A\otimes \Der(\A)$\\ with $\A=\R[\{\z_k,\z_\n\}]$\end{tabular}};
    \node (c6) [right=of c2']{\begin{tabular}{l} GpL-deformations\\ for $(\Der(\A),[\cdot,\cdot]_\circ)$:\\ $(\Der(\A),\td,-\mathbf{T}^{\nabla,[\cdot,\cdot]_\circ})$\end{tabular}};
    \node (c7) [below =of c6]{\begin{tabular}{l} M-K--L post-Lie alg.\\$(\X(\M),\nabla,-\mathbf{T}^{\nabla,[\cdot,\cdot]_J})$\end{tabular}};
    \node (c3') [right=of c4]{\begin{tabular}{l}Representation $\rhobar_{\btr}$\\ of $\env_{\lb\cdot,\cdot\rb}(L)$ on $\End(\A)$ \\ with $L\subset\A\otimes\Der(\A)$ \end{tabular}};
    \node (c5) [below=of c3']{\begin{tabular}{l} Recentering maps $\{\Gamma_f^{\btr}\}_{f\in G}$ \\  on $\Abar=\R[\{\z_k,\z_\n\}]$  \end{tabular}};
    \node (c3'')[right=of c3']{\begin{tabular}{l}Dual Hopf algebra\\ $(\env_{\lb\cdot,\cdot\rb}(L),\ast,\Delta_{\star_{\btr}})$
    \\and character group\\$(G,\star_{\btr},\ind)$\end{tabular}};
\draw[->] (c0.south) to (c1.north);
\draw[->] (c1.south west) to (c2.north);
\draw[->] (c2.south) to (c3);
\draw[->] (c2'.south) to (c3');
\draw[->] (c3.south) to (c4);
\draw[->] (c4.south east) to (c5);
\draw[->] (c3'.south) to (c5);
\draw[->] (c3''.south west) to (c5);
\draw[->] (c1.south east) to (c6.north);
\draw[->] (c6) -- (c7);
\draw[->] (c2.east) -- (c2');
\draw[->] (c3) -- (c3');
\draw[<->,dashed] (c3) to[bend right=10] (c7.west);
\draw[-latex] (c4.north east) to[bend left=20] (c3''.north west);
\draw[->] (c2'.south) to[bend right=17] (c3''.north);
\draw[->] (c6.south west) to (c3.east);
\end{tikzpicture}}

{\bf Acknowledgements}. The author kindly thanks Lorenzo Zambotti, Dominique Manchon, Kurusch Ebrahimi-Fard, Sylvie Paycha, Fabrizio Zanello and Martin Geller for insightful discussions on the topic of this paper.

\bigskip

\section{Torsion and curvature of a bilinear operation on a Lie algebra and conditions for post-Lie algebra structure.}\label{sec: torsion and curvature}
\subsection{Torsion and curvature of a bilinear operation on a Lie algebra.}

The aim of this first section is to generalize the geometric notions of torsion and curvature for general Lie algebras, endowed with a bilinear operation. We note that an algebraic approach of the notion of torsion and curvature has been used in \cite{gavriliov2008special} in the context of \textit{framed Lie algebras}.\\

Throughout this section, we adopt a broad perspective by considering an arbitrary Lie algebra $(L,[\cdot,\cdot])$. The term \textit{connection} on $L$ will refer to a bilinear operation:
\[\writefun{\td}{L\otimes L}{L}{(x,y)}{x\td y}\]
In differential geometry, when $L$ is the space of tangent vector fields on a smooth manifold, this operation is typically denoted $\nabla_{x}y$. However, following the notation used in \cite{gavriliov2008special}, we will use $x\td y$ instead, as it improves readability of the computations, without inducing any confusion.

\begin{definition}[\textbf{Torsion and Curvature tensors on a Lie algebra}]\label{def: torsion and curvature} Given a bilinear operation $\td\in \Hom(L^{\otimes 2},L)$, we associate two maps:
    \begin{itemize}
        \item $\textbf{T}^{\td,[\cdot,\cdot]}\in \Hom(\wedge^2 L,L)$ called \textbf{the torsion of} $\td$ \textbf{on} $(L,[\cdot,\cdot])$.
        \item $\textbf{R}^{\td,[\cdot,\cdot]} \in \Hom(L^{\otimes 3},L)$ called \textbf{the curvature of} $\td$ \textbf{on} $(L,[\cdot,\cdot])$.
    \end{itemize}
    which are defined by:
    \begin{align}
        \tors{\td,[\cdot,\cdot]}(x,y)&:=[x,y]_{\diamond}-[x,y]\label{eq: torsion}\\
        \curv{\td,[\cdot,\cdot]}(x,y,z)&:=x\td (y\td z)-y\td (x\td z)-[x,y]\td z\label{eq: curvature}
    \end{align}
    \end{definition}
We also remark by an easy computation, that there is an identity linking the curvature $\curv{\td,[\cdot,\cdot]}$, the torsion $\tors{\td,[\cdot,\cdot]}$, and the associator $\ass_{\td}$ of $\td$ which is given by:
\begin{equation}\label{eq: curvature vs torsion}
    \curv{\td,[\cdot,\cdot]}(x,y,z)=\ass_{\td}(x,y,z)-\ass_{\td}(y,x,z)+ \tors{\td,[\cdot,\cdot]}(x,y) \td z
\end{equation}

These definitions match with the geometric case where $\M$ is a finite dimensional smooth manifold and $(L,[\cdot,\cdot]):=(\Der(\C^\infty(\M,\R),[\cdot,\cdot]_\circ)$ is endowed with a connection $\nabla:\X(\M)\otimes \X(\M)\to \X(\M)$, we obtain the same notion of torsion and curvature.

Here we give in our context the classical notion of the covariant derivative of a $2$-fold tensor:

\begin{definition}[\textbf{Covariant derivative of a 2-fold tensor}]
    Let $\tors{}:L\otimes L\to L$ be a bilinear map. The $\td$-covariant derivative along $x\in L$, is the 2-fold tensor
    \[\writefun{\td \tors{}}{L^{\otimes 3}}{L}{x\otimes y\otimes z}{(x \td \tors{})(y,z)}\]
    defined by:
    \begin{equation}\label{eq: covariant derivative torsion}
        (x \td \tors{})(y,z):=x\td\left(\tors{}(y,z)\right) - \tors{}(x\td y,z) - \tors{}(y,x\td z)
    \end{equation}
    In particular, this definition will be applyed for the torsion: $\tors{}=\tors{\td,[\cdot,\cdot]}$
\end{definition}

\begin{remark}
    To define the notions of torsion and curvature, it is not necessary to assume that $[\cdot,\cdot]$ is a Lie bracket. However, we make this assumption from the begining, as $(L, [\cdot,\cdot])$ will be treated as a Lie algebra throughout the remainder of the paper.
\end{remark}

\subsection{Geometric post-Lie algebra deformation of a Lie algebra.}

We remind here the first Bianchi identity in our algebraical context, which is the key identity to understand the post-Lie conditions in a geometric perspective. A proof of it, inspired by \cite{nomizu1956lie}, can be found in the Appendix section \ref{sec: Appendix}. In order to write synthetically the first Bianchi identity, we introduce the following notation:

For a $3$-fold operator $\mathbf{A}\in\Hom(L^{\otimes 3},L)$, we denote using the symbol $\mathfrak{S}$ the operator on $\Hom(L^{\otimes 3},L)$ which sums over the cyclic permutations of the ordered set $(x,y,z)$, that is to say:
    \[\mathfrak{S}(\mathbf{A}(x,y,z))=\mathbf{A}(x,y,z)+\mathbf{A}(z,x,y)+\mathbf{A}(y,z,x)\]
\begin{lemma}[\textbf{The first Bianchi identity}]\label{lem: the first Bianchi identity}
Let $(L,[\cdot,\cdot])$ be a Lie algebra and let $\td:L\otimes L\to L$ be a bilinear operation. The following Bianchi equality is satisfied:
    \begin{equation}\label{eq: bianchi}
        \mathfrak{S}\left(\tors{\td,[\cdot,\cdot]}(\tors{\td,[\cdot,\cdot]} (x,y), z )\right)= \mathfrak{S}\left(\curv{\td,[\cdot,\cdot]}(x, y, z)\right) - \mathfrak{S}\left((x\td\tors{\td,[\cdot,\cdot]})(y,z)\right)
    \end{equation}
where we denoted again $\tors{\td,[\cdot,\cdot]}$ and $\curv{\td,[\cdot,\cdot]}$ respectively the torsion and the curvature defined by \eqref{eq: torsion} and  \eqref{eq: curvature}.
\end{lemma}

As a direct corollary of the first Bianchi identity, we have the following Theorem, which give sufficient conditions on the connection $\td$ to ensure the geometric Post-Lie deformation on $(L,[\cdot,\cdot])$:
\begin{theorem}[Geometric post-Lie deformation of a Lie algebra]\label{th: geometric post-Lie algebra} Let $(L,[\cdot,\cdot])$ be a Lie algebra and let $\td:L\otimes L\to L$ be a bilinear operation with:
\begin{itemize}
    \item null curvature $\curv{\td,[\cdot,\cdot]}\equiv 0$,
    \item constant torsion $\td \tors{\td,[\cdot,\cdot]}\equiv 0$,
\end{itemize}
then $(L,\td,-\tors{\td,[\cdot,\cdot]})=(L,\td,[\cdot,\cdot]-[\cdot,\cdot]_{\diamond})$ is a post-Lie algebra, which is called the \textbf{$\td$-geometric post-Lie deformation of $(L,[\cdot,\cdot])$}, abbreviated \textbf{$\td$-gpL deformation of $(L,[\cdot,\cdot])$}.
\end{theorem}
\begin{proof}
    Let $x,y,z\in L$. By definition of the covariant derivative of the torsion \eqref{eq: covariant derivative torsion}, setting $x\td\tors{\td,[\cdot,\cdot]}=0$, one gets the following equality:
    \[x\td\tors{\td,[\cdot,\cdot]}(y,z)=\tors{\td,[\cdot,\cdot]}(y,x\td z)+\tors{\td,[\cdot,\cdot]}(x\td y,z)\]
    Then, if the curvature is null $\curv{\td,[\cdot,\cdot]}\equiv 0$, we get from the equality \eqref{eq: curvature vs torsion} that:
    \[-\tors{\td,[\cdot,\cdot]}(x,y) \td z=\ass_{\td}(x,y,z)-\ass_{\td}(y,x,z)\]
    Finally, the null curvature and constant torsion hypothesis make the right-hand side term of the first Bianchi equality \eqref{eq: bianchi} vanish and we get:
    \[\mathfrak{S} \left(\tors{\td,[\cdot,\cdot]}\lp\tors{\td,[\cdot,\cdot]} (x,y),z\rp\right)=0\]
    which is the Jacobi equality. Hence $-\tors{\td,[\cdot,\cdot]}:=[\cdot,\cdot]-[\cdot,\cdot]_{\diamond}$ is a Lie bracket.
\end{proof}

\begin{remark}
    A pre-Lie algebra $(L,\tr)$ is the $\tr$-gpL deformation of the Lie algebra $(L,[\cdot,\cdot]_{\tr})$.
\end{remark}

In a pre-Lie algebra $(L, \tr)$, the commutator $[\cdot,\cdot]_\tr$ satisfies the Jacobi identity, and therefore defines a Lie bracket. In contrast, in a post-Lie algebra $(L,\tr, [\cdot,\cdot])$, the commutator $[\cdot,\cdot]_\tr$ does not, in general, satisfy the Jacobi identity. Nevertheless, the following result holds:
\begin{prop}[\cite{ebrahimi2014lie}] \label{prop: composition Lie bracket}
Let $(L,\tr,[\cdot,\cdot])$ be a post-Lie algebra. The bilinear operation $\llbracket \cdot,\cdot \rrbracket:L\otimes L\rightarrow L$ given by:
\begin{equation} \label{eq: composition Lie bracket}
    \llbracket\cdot,\cdot \rrbracket: = [\cdot,\cdot]_\tr + [\cdot,\cdot]
\end{equation}
is a Lie bracket.
\end{prop}

\begin{remark}
    Note that any post-Lie algebra $(L,\tr,[\cdot,\cdot])$, can be seen as the $\tr$-gpL-deformation of the Lie algebra $(L,\llbracket\cdot,\cdot \rrbracket)$, indeed in that case: 
    \[-\tors{\tr,\llbracket\cdot,\cdot \rrbracket}=\llbracket\cdot,\cdot \rrbracket - [\cdot,\cdot]_\tr= [\cdot,\cdot]\]
\end{remark}

\subsection{The case of the Lie algebra of derivations on a commutative and associative algebra.}
Let $(\A,\cdot)$ be an associative and commutative algebra. The space of derivations, denoted $\Der(\A,\cdot)$ (or simply $\Der(\A)$ when no confusions occur), is the subspace of all $D\in\End(\A)$ satisfying the following Leibniz rule:
\begin{equation}
    D(a\cdot b)=D(a)\cdot b + a\cdot D(b)
\end{equation}
which formula can be generalised inductively for all $a_1,\ldots,a_n\in\A$ into:
\begin{equation*}
    D(a_1\cdots a_n)=\sum_{i=1}^n a_1\cdots D(a_i)\cdots a_n.
\end{equation*}
By associativity of the compositon product $\circ$, we have that $[\cdot,\cdot]_\circ$ verifies the Jacobi identity, and we have also that it stabilizes $\Der(\A)$, that is to say:
\[[D_1,D_2]_\circ(a \cdot b)=[D_1,D_2]_\circ(a)\cdot b + a\cdot [D_1,D_2]_\circ(b),\qquad a,b\in\A\]
which proves that $[D_1,D_2]_\circ\in \Der(\A)$.\\

A particular case of interest where the last Theorem \ref{th: geometric post-Lie algebra} applies is when $(L,[\cdot,\cdot])=(\Der(\A),[\cdot,\cdot]_\circ)$. We obtain in that case a post-Lie algebra structure, given in the following corollary:

\begin{corollary}[of Theorem \ref{th: geometric post-Lie algebra}]
    Let $\A$ be a commutative and associative algebra, and let $\td:\Der(\A)\otimes \Der(\A)\to \Der(\A)$ be a bilinear operation. If $\td$ has null curvature $\curv{\td,[\cdot,\cdot]_\circ}\equiv0$ and constant torsion $\td \tors{\td,[\cdot,\cdot]_\circ}\equiv0$ on the Lie algebra $(\Der(\A),[\cdot,\cdot]_\circ)$, then  $(\Der(\A),\td,[\cdot,\cdot]_\circ-[\cdot,\cdot]_{\diamond})$ is a post-Lie algebra.
\end{corollary}

In differential geometry, given a finite dimensional smooth manifold $\M$, and considering the (commutative and associative) algebra of smooth maps on the manifold $\C^\infty(\M,\R)$, endowed with the pointwise product defined as
\[(f\cdot g)(m)=f(m)g(m),\qquad m\in\M;\]
the space of smooth tangential vector fields $\X(\M)$ can be defined as a space of derivation (see for example \cite[Proposition 8.15]{lee2003smooth}):
\[\X(\M):=\Der(\C^\infty(\M,\R)).\]
Hence, considering the particular case othe the last Corollary where $\A:=\C^\infty(\M,\R)$ and $\Der(\A):=\X(\M)$, endowed with a connection \[\nabla:\X(\M)\otimes \X(\M)\to \X(\M),\] we obtain as a direct corollary the following:
\begin{corollary}[Munthe-Kaas Lundervolt conditions for Post-Lie structure \cite{munthe2013post}]\label{theo: munthe-kaas lundervold} Given finite dimensional smooth manifold $\M$ and a connection $\nabla$ on the smooth sections of the tangential vector bundle $\X(\M)$, if $\nabla$ is flat and has constant torsion, then denoting $[\cdot,\cdot]_J$ the Jacobi-Lie bracket of vector fields, we obtain that $(\X(\M),\nabla,-\mathbf{T}^{\nabla,[\cdot,\cdot]_J})$ is a post-Lie algebra.
\end{corollary}
\medskip

\subsection{Geometric post-Lie (gpL) deformation of a post-Lie algebra.}

In that subsection, we define a class of deformations for post-Lie algebras and we derive sufficient compatibility conditions for the deformation to be a post-Lie algebra.

Let us begin with a theorem that is central in our paper and that is the starting point of the study of gpL-deformations applying for the general setting of post-Lie algebras, which is a generalisation of the previous Theorem \ref{th: geometric post-Lie algebra}:
\begin{theorem}\label{th: deformation post-Lie}
    Let $(L,\tr,[\cdot,\cdot])$ be a post-Lie algebra and let $\td:L\otimes L \to L$ be a bilinear operation. If the following compatibility conditions are satisfied:
    \begin{itemize}
        \item Compatibility between $\td$ and $[\cdot,\cdot]$: 
        \[\td\tors{\td,[\cdot,\cdot]}\equiv 0\qquad \text{and}\qquad \curv{\td,[\cdot,\cdot]}\equiv 0\]
        \item Compatibility between $\td$ and $\tr$:
        \begin{equation}\label{eq: compatibility derivation}
            x \tr\cdot ~\in\Der(L,\td)
        \end{equation}
    \end{itemize}
    then $(L,\btr,\lb\cdot,\cdot\rb)$ is a post-Lie algebra, with the notations:
    \begin{equation}\label{eq: notations}
        \btr:=\tr+\td\quad \text{and}\quad \lb\cdot,\cdot\rb:=[\cdot,\cdot] - [\cdot,\cdot]_{\diamond}=-\tors{\td,[\cdot,\cdot]}
    \end{equation} 
    We say that $(L,\btr,\lb\cdot,\cdot\rb)$ is the {\bf$\td$-gpL-deformation} of $(L,\tr,[\cdot,\cdot])$
\end{theorem}

\begin{remark}
    In the setting of the previous Theorem \ref{th: deformation post-Lie}, applying the Theorem \ref{th: geometric post-Lie algebra} to the Lie algebra $(L,[\cdot,\cdot])$, we also have that $(L,\td,\lb\cdot,\cdot\rb)$ is a post-Lie algebra, giving a second post-Lie structure on $L$ with the same Lie bracket component $\lb\cdot,\cdot\rb$.
\end{remark}
 
\begin{remark}
    GpL-deformations for pre-Lie algebras $(L,\tr)$ are of the form $(L,\tr+\td,-[\cdot,\cdot]_{\diamond})$ and they are therefore of the same type as the deformations in \cite{bruned2025post}[Theorem 2.4] where $\omega=\td$ and $\pi=-[\cdot,\cdot]_{\diamond}$.
\end{remark}


For sake of clarity, we split the proof of the Theorem in the three following Lemmas, whose proofs are relegated to the appendix section for greater clarity:
\begin{lemma}\label{lem: Der alg. ==> Der Lie alg.} Let $(L,\tr,\td)$ be a linear space, endowed with two bilinear operation $\tr,\td\in \Hom(L^{\otimes 2},L)$. For all $x\in L$:
    $$ x\tr \cdot \in \Der(L,\td) \Rightarrow  x\tr\cdot \in \Der(L,[\cdot,\cdot]_{\diamond})$$
\end{lemma}

\begin{lemma}[\textbf{The constant torsion hypothesis}]\label{lemma: constant torsion post-Lie identity} Let $(L,\tr,[\cdot,\cdot])$ be a post-Lie algebra, and let $(L,\btr,\lb\cdot,\cdot\rb)$ be defined as in Theorem \ref{th: deformation post-Lie}. If the conditions $\td\tors{\td,[\cdot,\cdot]}\equiv 0$ and \eqref{eq: compatibility derivation} are satisfied, then we have the following equality for all $x,y,z\in L$:
\begin{equation}\label{eq: post-Lie geometric derivation alg. Leibniz}
    x\tr\lb y,z\rb=\lb x\tr y,z\rb+\lb y,x\tr z\rb
\end{equation}
\end{lemma}

\begin{lemma}[\textbf{Curvature identity}]\label{lemma: curvature post-Lie identity} Let $(L,\tr,[\cdot,\cdot])$ be a post-Lie algebra, and let $(L,\btr,\lb\cdot,\cdot\rb)$ be defined as in Theorem \ref{th: deformation post-Lie}. If the compatibility condition \eqref{eq: compatibility derivation} is satisfied, then we have the following identity for all $x,y,z\in L$:
\begin{equation}
    \ass_{\btr}(x,y,z)-\ass_{\btr}(y,x,z)=\lb x,y\rb \btr z + \curv{\td,[\cdot,\cdot]}[x,y,z]
\end{equation}
\end{lemma}

Finaly, from the lasts two Lemmas, one easily deduces the proof of the Theorem \ref{th: deformation post-Lie}:
\begin{proof}[Proof the Theorem \ref{th: deformation post-Lie}] We prove easily that the conditions of Definition \ref{def: post-Lie} of post-Lie algebras are satisfied using the previous Lemmas:
\begin{itemize}
    \item By Lemma \ref{lemma: constant torsion post-Lie identity}, with $\td\tors{\td,[\cdot,\cdot]}\equiv0$, the condition \eqref{eq: post-Lie derivation condition} is satisfied.
    \item By Lemma \ref{lemma: curvature post-Lie identity}, with $\curv{\td,[\cdot,\cdot]}\equiv 0$, the condition \eqref{eq: post-Lie associator condition} is satisfied.
    \item By the first Bianchi identity \ref{lem: the first Bianchi identity} with both conditions $\td\tors{\td,[\cdot,\cdot]}\equiv 0$ and $\curv{\td,[\cdot,\cdot]}\equiv 0$, the bracket $\lb\cdot,\cdot\rb$ satifies the Jacobi identity since:
    \[\mathfrak{S}\left(\tors{\td,[\cdot,\cdot]}(\tors{\td,[\cdot,\cdot]} (x,y), z )\right)=\mathfrak{S}\Big(\lb\lb x,y\rb, z \rb\Big)=0\]
    We also trivially have the anti-symmetry condition, and hence $\lb\cdot,\cdot\rb$ is a Lie bracket on $L$.
\end{itemize}
\end{proof}

\begin{corollary}[Adjoint post-Lie algebras, see \cite{munthe2013post} Proposition 2.6.]
    Let $(L,\tr, [\cdot,\cdot])$ be a post-Lie algebra and define the product $\btr$ as:
\[x \btr y := x \tr y + [x, y].\]
Then, the adjoint $(L, \btr, -[\cdot,\cdot])$ of $(L,\tr, [\cdot,\cdot])$ is also a post-Lie algebra. The operation consisting in taking the adjoint of a post-Lie algebra is an involution.
\end{corollary}
\begin{proof}
    We apply the geometric deformation of Theorem \ref{th: deformation post-Lie} with $\td:=[\cdot,\cdot]$ on the post-Lie algebra $(L,\tr,[\cdot,\cdot])$. Let us verify that the compatibility conditions are satisfied:\\
    By anti-symmetricity of $[\cdot,\cdot]$, we trivially have that the torsion of $[\cdot,\cdot]$ on $(L,[\cdot,\cdot])$ is given by $\mathbf{T}^{[\cdot,\cdot],[\cdot,\cdot]}=-[\cdot,\cdot]$, and its $[\cdot,\cdot]$-covariant derivative, given by the equality \eqref{eq: covariant derivative torsion} is therefore null $[\cdot,\cdot]\td\tors{[\cdot,\cdot],[\cdot,\cdot]}\equiv 0$ because of the Jacobi identity. To prove that the curvature is null, using the anti-symmetricity of $[\cdot,\cdot]$ and the Jacobi identity:
    \begin{align*}
        \curv{[\cdot,\cdot],[\cdot,\cdot]}(x,y,z)&:=[x,[y, z]]-[y,[x, z]]-[[x,y], z]\\
        &=[x,[y, z]]+[y,[z, x]]+[z,[x,y]]\\
        &=0
    \end{align*}
\end{proof}

\subsection{Hopf algebra structure on the Lie enveloping algebra of a gpL-deformation of a post-Lie algebra.}

The Lie enveloping algebra of a Lie algebra $(L,[\cdot,\cdot])$, denoted $\env_{[\cdot,\cdot]}(L)$, is defined as the tensor algebra $\tens(L)=\bigoplus_{k\geq 0} L^{\otimes k}$ over $L$ quotiented by the two-sided ideal $\mathfrak{c}$ generated by $\{x\otimes y - y\otimes x - [x,y]:\,x,y\in L\}$:
\[\env_{[\cdot,\cdot]}(L):=\tens(L)/ \mathfrak{c}.\]
As no confusion is likely to arise, we will use the same notation \( x_1 \cdots x_n \) for both the equivalence class in \( \env_{[\cdot,\cdot]}(L) \) and its representative in \( \tens(L) \). The following theorem allows one to describe a basis on $\env_{[\cdot,\cdot]}(L)$:
\begin{theorem}[Poincaré-Birkhoff-Witt]\label{theo: PBW}
    Given a basis $\B_L$ of $L$ and a total order $\leq$ on it, a basis $\B_{\envU(L)}=\B_{\envU(L)}^\leq$ of $\env_{[\cdot,\cdot]}(L)$ is given by 
    \begin{equation}\label{eq:basis}
\begin{split}
\B_{\envU(L)}:=&\{\mathds 1\}\sqcup\Bigg\{\frac1{m_1!\cdots m_k!}\,x_1^{m_1}\cdots x_k^{m_k}: \   k, m_1,\ldots,m_k\geq 1, \\
& \qquad \qquad x_1<\ldots<x_k, \ x_i\in\B_L
\Bigg\}.
\end{split}
\end{equation}
\end{theorem}

We denote by $\Delta_\ast$ the \textit{coshuffle coproduct} which is defined on $\env_{[\cdot,\cdot]}(L)$, using the basis $\B_{\envU(L)}$ in \eqref{eq:basis} as:
\begin{equation}\label{eq:cop1}
\begin{split}
\Delta_\ast \prod_{i=1}^k \frac{x_i^{m_i}}{m_i!} = \prod_{i=1}^k \Delta_\ast  \frac{x_i^{m_i}}{m_i!}
&= \prod_{i=1}^k \sum_{\ell=0}^{m_i} \frac{x_i^{\ell}}{\ell!}\otimes\frac{x_i^{m_i-\ell}}{(m_i-\ell)!}
\\ & = \sum_{0\le\ell_i\le m_i} \left(\prod_{i=1}^k\frac{x_i^{\ell_i}}{\ell_i!}\right)\otimes
\left(\prod_{i=1}^k\frac{x_i^{m_i-\ell_i}}{(m_i-\ell_i)!}\right),
\end{split}
\end{equation}
for which we recall Sweedler's notation:
\begin{equation}\label{eq: Sweedler}
    \Delta_\ast(u)=\sum_{\Delta_\ast(u)} u^{(1)} \otimes u^{(2)},
\end{equation}

Let us now recall the extension of the product $\tr$ to all $u,v \in \env_{[\cdot,\cdot]}(L)$, see Proposition 3.1 in \cite{ebrahimi2014lie}:
\begin{prop}
\label{prop: extension post-lie product}
Let $(L,\tr,[\cdot,\cdot])$ be a post-Lie algebra. There exists a unique extension of the product $\tr$ to $\env_{[\cdot,\cdot]}(L)$ which verifies for all $x\in L$ and $u,v,w\in \env_{[\cdot,\cdot]}(L)$:
\begin{enumerate}
    \item $\1 \tr u = u$, $u\tr \1=\varepsilon(u) \1$
    \item $(xv)\tr w=x\tr(v\tr w)-(x\tr v)\tr w$
    \item $u\tr (v w)=\displaystyle\sum_{\Delta_\ast(u)} (u^{(1)}\tr v)(u^{(2)}\tr w)$.
\end{enumerate}
\end{prop}

The next Proposition \ref{prop: associative product} has been proved in \cite[Proposition 3.3]{ebrahimi2014lie}, which extends the Guin--Oudom approach \cite{oudom2008lie}, originally used in the case of pre-Lie algebras, to the case of post-Lie algebras. 

\begin{prop} \label{prop: associative product} 
Let $(L,\tr,[\cdot,\cdot])$ be a post-Lie algebra. The product $\star_{\tr}:\env_{[\cdot,\cdot]}(L)\otimes \env_{[\cdot,\cdot]}(L)\to \env_{[\cdot,\cdot]}(L)$ defined for all $u,v\in \env_{[\cdot,\cdot]}(L)$ using Sweedler's notation, by:
\begin{equation}
\label{eq: post-Lie associative product}
	u\star_{\tr} v =\displaystyle\sum_{\Delta_\ast(u)} u^{(1)} (u^{(2)}\tr v)
\end{equation}
is associative and $\Big(\env_{[\cdot,\cdot]}(L),\star_{\tr},\Delta_\ast,\ind,\varepsilon\Big)$ is a Hopf algebra, where the unit element $\ind$ is given as the canonical injection map $\R\ni t\mapsto
t\ind\in\env_{[\cdot,\cdot]}(L)$ and the counit $\varepsilon$ is given as the map $\env_{[\cdot,\cdot]}(L)\ni x\mapsto \varepsilon(x)\in\R$ where 
$x-\varepsilon(x){\ind}\in \bigoplus_{k\geq 1} L^{\otimes k}/{\mathfrak c}$.
\end{prop}

In particular we have the following usefull equalities for $x_0,x_1,\ldots,x_n\in L$:
\begin{equation}\label{eq:atrbbb}
\begin{split}
x_0\star_{\tr} (x_1\cdots x_n) & = x_0\tr (x_1\cdots x_n)+ x_0x_1\cdots x_n
\\ & = \sum_{i=1}^n x_1 \cdots (x_0\tr x_i) \cdots x_n+ x_0x_1\cdots x_n,
\end{split}
\end{equation}

Then, the associativity of $\star_\tr$ implies that the commutator $[\cdot,\cdot]_{\star_{\tr}}$ is a Lie bracket on $\env_{[\cdot,\cdot]}(L)$; moreover, we have for all $x,y\in L\subset \env_{[\cdot,\cdot]}(L)$ that \[[x,y]_{\star_{\tr}}:=x\tr y -y\tr x+ xy-yx=[x,y]_{\tr}+ [x,y]=\llbracket x,y \rrbracket,\] where the Lie bracket $\llbracket \cdot,\cdot \rrbracket$ has been defined in Proposition \ref{eq: composition Lie bracket}. Therefore, the canonical injection $L\hookrightarrow \env_{[\cdot,\cdot]}(L)$ is a Lie algebra morphism $(L,\llbracket \cdot,\cdot \rrbracket)\hookrightarrow (\env_{[\cdot,\cdot]}(L),[\cdot,\cdot]_{\star_{\tr}})$, which factors through the Lie enveloping algebra $\env_{\llbracket \cdot,\cdot \rrbracket}(L)$. Denoting by $\conc$, respectively $\Delta_\ast$, the concatenation product, respectively the unshuffle coproduct, on $\env_{\llbracket \cdot,\cdot \rrbracket}(L)$, we recall the following:

\begin{theorem}[\cite{jacques2023post} Theorem 2.9.]\label{theo: isomorphism between enveloping algebras}
Let $(L,\tr,[\cdot,\cdot])$ be a post-Lie algebra, and let $\llbracket\cdot,\cdot \rrbracket:= [\cdot,\cdot]_\tr + [\cdot,\cdot]$, denote the composition Lie bracket on $L$, given in Proposition \ref{prop: composition Lie bracket}.\\
The linear map $\Phi_{\tr}$ defined below is an isomorphism of Hopf algebras:
$$\writefuniso{\Phi_{\tr}}{\left(\env_{\llbracket \cdot,\cdot \rrbracket}(L),\mathsf{conc},\Delta_\ast\right)}{\left(\env_{[\cdot,\cdot]}(L),\star_{\tr},\Delta_\ast\right)}{x_1 \cdots x_n}{x_1 \star_{\tr} \cdots \star_{\tr} x_n}\qquad x_1,\ldots,x_n\in L$$
\end{theorem}

Now, fix a post-Lie algebra $(L,\tr,[\cdot,\cdot])$ and consider a $\td$-gpL-deformation $(L,\btr,\lb\cdot,\cdot\rb)$, defined by Theorem \ref{th: deformation post-Lie}, where we denote again $\btr:=\tr+\td$ and $\lb\cdot,\cdot\rb:=[\cdot,\cdot] + [\cdot,\cdot]_{\diamond}$. We remark that we also have
$$\llbracket\cdot,\cdot \rrbracket:= [\cdot,\cdot]_{\btr} + \lb\cdot,\cdot\rb;$$
indeed, for all $x,y\in L$:
\begin{align*}
    \llbracket x,y\rrbracket:=[x,y]_\tr + [x,y]&=([x,y]_\tr + [x,y]_{\diamond}) + ([x,y] - [x,y]_{\diamond})\\
    &=[x,y]_{\btr} + \lb x,y\rb
\end{align*}
Thus, applying the Guin-Oudom extension machinery of Theorems \ref{prop: extension post-lie product} and \ref{prop: associative product} to $(L,\btr,\lb\cdot,\cdot\rb)$, we obtain an associative Hopf algebra $(\env_{\lb\cdot,\cdot\rb}(L),\star_{\btr},\Delta_\ast)$ and Theorem \ref{theo: isomorphism between enveloping algebras} above shows the existence of an isomophism of Hopf algebras $\Phi_{\btr}$ as given in the Proposition below:

\begin{prop} Given a post-Lie algebra $(L,\tr,[\cdot,\cdot])$ and a gpL deformation $(L,\btr,\lb\cdot,\cdot\rb)$, where $\btr:=\tr+\td$ and $\lb\cdot,\cdot\rb:=[\cdot,\cdot] + [\cdot,\cdot]_{\diamond}$, we have an isomophism of Hopf algebras $\Phi_{\btr}$ given by:
\[\writefuniso{\Phi_{\btr}}{\left(\env_{\llbracket \cdot,\cdot \rrbracket}(L),\mathsf{conc},\Delta_\ast\right)}{\left(\env_{\lb\cdot,\cdot\rb}(L),\star_{\btr},\Delta_\ast\right)}{x_1 \cdots x_n}{x_1 \star_{\btr} \cdots \star_{\btr} x_n}\]
As a result, by composition we obtain the following isomophism of Hopf algebras:
\begin{equation}\label{eq: iso envl alg}
    \writefuniso{\Phi_{\tr}^{-1}\circ \Phi_{\btr}}{\left(\env_{[\cdot,\cdot]}(L),\star_{\tr},\Delta_\ast\right)}{\left(\env_{\lb\cdot,\cdot\rb}(L),\star_{\btr},\Delta_\ast\right)}{x_1 \star_{\tr} \cdots \star_{\tr} x_n}{x_1 \star_{\btr} \cdots \star_{\btr} x_n}
\end{equation}
\end{prop}

The following commutative diagram summarises the lasts two Propositions:
\begin{equation}\label{diag: envl alg}
\begin{tikzcd}[scale cd=1,sep=large]
    & \lp \env_{\llbracket \cdot,\cdot \rrbracket}(L),\mathsf{conc}, \Delta_\ast\ \rp \arrow["\sim"']{dr}{\Phi_{\btr}} \\
    \left(\env_{[\cdot,\cdot]}(L),\star_{\tr}, \Delta_\ast\right) \arrow["\sim"']{ur}{\Phi_{\tr}^{-1}}\arrow[rr,"\sim", "\Phi_{\tr}^{-1}\circ \Phi_{\btr}"'] && \left(\env_{\lb\cdot,\cdot\rb}(L),\star_{\btr}, \Delta_\ast\right)
  \end{tikzcd}
\end{equation}

\subsection{Pre-Lie deformations.}

In that subsection, we derive conditions such that a gpL-deformation can turn a post-lie algebra to turn a post-Lie algebra $(L,\tr,[\cdot,\cdot])$ into a pre-Lie algebra $(L,\btr)$. An example for such a pre-Lie deformation will be given later in Theorem \ref{theo: pre-Lie deformation RG} in the context of regularity structures. We recall that in the case if $\lb\cdot,\cdot\rb\equiv 0$, the monomials of the enveloping algebra $\env_{\lb\cdot,\cdot\rb}(L)$ are symmetric, which means that $\env_{\lb\cdot,\cdot\rb}(L)=\sym(L)$ the \textbf{symmetric algebra} on $L$, which is the symmetrization of $\tens(L)=\bigoplus_{n\in \N} L^{\otimes n}$ defined by quotienting it by the bilateral ideal generated by the elements $x\otimes y - y\otimes x$. 

\begin{prop}[Pre-Lie deformation]\label{prop: pre-Lie deformation}
    Let $(L,\td)$ be a pre-Lie algebra and let $\tr:L\otimes L\to L$ be a bilinear operation, such that $x\tr \cdot\in Der(L,\td)$ then:
    \[(L,\tr,[\cdot,\cdot]_{\diamond})~\text{is a post-Lie algebra} \Rightarrow (L,\btr=\tr + \td)~\text{is a pre-Lie algebra}\]
    The reverse implication is garanteed if we suppose moreover that $[x,y]_{\diamond}\td z=0$ for all $x,y,z\in L$. Moreover, as in \eqref{eq: iso envl alg}, there is an isomophism of Hopf algebras:
    $$\Phi_{\tr}^{-1}\circ \Phi_{\btr}:(\env_{[\cdot,\cdot]_{\diamond}}(L),\star_\tr,\Delta_\ast)\xrightarrow{\sim}(\sym(L),\star_{\btr},\Delta_\ast)$$
\end{prop}
\begin{proof} To show the implication, we apply a $\td$-gpL deformation of Theorem \ref{th: deformation post-Lie} on the post-Lie algebra $(L,\tr,[\cdot,\cdot])$. Let us verify that the compatibility conditions are satisfied:\\
By definition \ref{def: torsion and curvature} of the torsion of a bilinear operation on a Lie algebra, we trivially have that the torsion of $\td$ on $(L,[\cdot,\cdot])$ is given by $\mathbf{T}^{\td,[\cdot,\cdot]}=\0$, and its covariant derivative, given by the equality \eqref{eq: covariant derivative torsion} is therefore null because of the Jacobi identity. To compute the curvature, we use the equality \eqref{eq: curvature vs torsion} where the torsion term vanishes, which aims to the following equality, for all $x,y,z\in L$:
\[\mathbf{R}^{\td,[\cdot,\cdot]}(x,y,z)=\ass_{\td}(x,y,z) - \ass_{\td}(y,x,z)=0,\]
by the definition of a pre-Lie product $\td$, see equality \eqref{eq: pre-Lie equality}.\\
Now, to show the reverse implication (under the extra hypothesis), denoting by $\0$ the null bilinear operation, we apply a $(-\td)$-gpL deformation on the post-Lie algebra $(L,\tr + \td,\0)$:\\
The torsion of $-\td$ on $(L,\0)$ is simply equal to $\mathbf{T}^{-\td,\0}=-[\cdot,\cdot]_{\diamond}$ and the Lemma \ref{lem: Der alg. ==> Der Lie alg.} indicates that the covariant derivative of the torsion is null $\td\mathbf{T}^{-\td,\0}=\0$. Then, for the curvature, using \eqref{eq: curvature vs torsion}:
\[\mathbf{R}^{-\td,\0}(x,y,z)=\ass_{\td}(x,y,z) - \ass_{\td}(y,x,z) + [x,y]_{\diamond}\td z=0\]
\end{proof}

The advantage of using the symmetric tensor algebra $\sym(L)$ is that given a basis $\B_L$ of $L$, it has a canonical basis formed by the monomials:
\[\ind \cup \{x_1\cdots x_n,~n\geq 1,~(x_1,\ldots, x_n)\in \B_L^n\}\]
Therefore, using the symmetric tensor algebra $\sym(L)$, instead of using the enveloping algebra $\envU(L)$, avoids the need of defining a total order on the basis elements $\B_L$ of $L$ to apply the Poincaré-Birkhoff-Witt theorem for defining a linear basis for $\envU(L)$, which is an arbitrary choice, see \cite{jacques2023post}[Theorem 2.5]. However, in some practical construction of such a product $\tr$, a total order on non-commuting basis elements of $\B_L$ is used, see remark \ref{rem: construction pre-Lie product} below:

\begin{remark}\label{rem: construction pre-Lie product}
    Given a post-Lie algebra $(L,\tr,[\cdot,\cdot])$, in order to apply the preceding Proposition, a strategy is to try to find a pre-Lie product $\td:L\otimes L\to L$ whose commutator gives back the Lie bracket $[\cdot,\cdot]_{\diamond}=[\cdot,\cdot]$. It can be made by fixing a linear basis $\B_L$ of $L$ for an index set $I$ and defining a strict total order relation $<$ on $\{(x,y)\in\B_L^2,~[x,y]\neq 0\}$. Then we can define $\td:L\otimes L\to L$ on $\B_L$ as:
    \begin{equation}\label{eq: pre-Lie product in coordinates}
        x\td y := \begin{cases}
        [x,y]\quad\text{if}\quad [x,y]\neq 0~\wedge~ x<y\\
        0\qquad\qquad\text{else}
    \end{cases}
    \end{equation}
    Note that such situation occurs in regularity structures (see Subsection \ref{subsec: Pre-Lie regularity structures} below).
\end{remark}

\section{Geometric post-Lie deformation of the J--Z post-Lie algebra and applications.}\label{sec: deformation post-Lie alg of derivations}

In that section we consider a commutative and associative algebra $(\A,\cdot)$ and we will endow $L:=\A\otimes \Der(\A)$ with an algebraic structure given by two bilinear operations $(\btr,\lb\cdot,\cdot\rb)$, which generalizes the one given in \cite{jacques2023post} and satisfy the post-Lie property under certain hypothesis.

We recall that in \cite{jacques2023post}, the following canonical post-Lie structure has been derived:
\begin{theorem}[\cite{jacques2023post} Theorem 3.1]\label{theo: post-Lie structure from derivations}
Defining on the space $L:=\A\otimes \Der(\A)$ two bilinear operations $\tr,[\cdot,\cdot]:L\otimes L\to L$ given for all $a_1,a_2\in\A$ and $D_1,D_2\in \Der(\A)$ by: 
\begin{equation}\label{eq: post-lie product derivation}
    a_1\otimes D_1\tr a_2 \otimes D_2:=a_1 D_1(a_2)\otimes D_2,
\end{equation}
\begin{equation}\label{eq: post-lie bracket derivation}
    [a_1\otimes D_1, a_2 \otimes D_2]:=a_1a_2\otimes [D_1,D_2]_\circ.
\end{equation}
then $(L,\tr,[\cdot,\cdot])$ is a (left) post-Lie algebra. 
\end{theorem}

\subsection{Torsion and curvature on \texorpdfstring{\(\A\otimes \Der(\A)\)}{} and geometric post-Lie algebra structures.}\label{subsec: pL alg of derivations}
In all of that subsection, we denote by \(\D\subset\Der(\A)\) a subspace of \(\Der(\A)\), which is stable by the commutator of the composition product: \[[\D,\D]_\circ\subset \D.\]
The space $\A\otimes \D$ can be endowed with a Lie bracket $[\cdot,\cdot]$, defined for all $a_1,a_2\in \A$ and $D_1, D_2\in \D$ by:
\begin{equation}
    [a_1\otimes D_1,a_2\otimes D_2]:=a_1a_2\otimes [D_1,D_2]_\circ
\end{equation}
which trivially turns $(\A\otimes \D,[\cdot,\cdot])$ into a Lie algebra:

In a similar way, a bilinear operation $\td\in \Hom(\D^{\otimes 2},\D)$ on $\D$ can be lifted on $\A\otimes \D$ into a bilinear operation $\td\in \Hom((\A\otimes \D)^{\otimes 2},\A\otimes \D)$ (using the same symbol, since it does not induce any confusion), which is defined for all $a_1,a_2\in \A$ and $D_1, D_2\in \D$ by:
\begin{equation}\label{eq: extension connections on derivations}
    (a_1\otimes D_1)\td(a_2\otimes D_2):=a_1a_2\otimes (D_1 \td D_2)
\end{equation}

Denoting again by $\tors{\td,[\cdot,\cdot]}$ the torsion and by $\curv{\td,[\cdot,\cdot]}$ the  curvature of $\td$ on the Lie algebra $(\A\otimes\D,[\cdot,\cdot])$, we obtain for all $a_1,a_2,a_3\in\A$ and all $D_1,D_2,D_3\in \D$ that :
\begin{align*}
    \tors{\td,[\cdot,\cdot]}(a_1\otimes D_1,a_2\otimes D_2)&=a_1a_2\otimes \tors{\td,[\cdot,\cdot]_\circ}(D_1,D_2)\\
    &=a_1a_2\otimes ([D_1,D_2]_{\diamond}-[D_1,D_2]_\circ)\\
    \curv{\td,[\cdot,\cdot]}(a_1\otimes D_1,a_2\otimes D_2,a_3\otimes D_3)&=a_1a_2a_3\otimes \curv{\td,[\cdot,\cdot]_\circ}(D_1,D_2,D_3)
\end{align*}
\medskip
Schematically speaking, we have the following transport of structure:
\begin{equation*}
\boxed{
\begin{aligned}
    \underset{[D_1,D_2]_\circ}{\text{Commutator $[\cdot,\cdot]_\circ$ on $\D$}}~&\Longrightarrow~ \underset{[a_1\otimes D_1,a_2\otimes D_2]=a_1a_2\otimes [D_1,D_2]_\circ}{\text{Lie bracket $[\cdot,\cdot]$ on $\A\otimes \D$}}\\
    \underset{D_1 \td D_2}{\text{Connection $\td$ on $\D$}}~&\Longrightarrow~ \underset{(a_1\otimes D_1)\td(a_2\otimes D_2):=a_1a_2\otimes D_1 \td D_2}{\text{Connection $\td$ on $\A\otimes \D$}}\\
    \underset{\tors{\td,[\cdot,\cdot]_\circ}(D_1,D_2)}{\text{Torsion $\tors{\td,[\cdot,\cdot]_\circ}$ on $\D$}} ~&\Longrightarrow ~ \underset{\tors{\td,[\cdot,\cdot]}(a_1\otimes D_1,a_2\otimes D_2)=a_1a_2\otimes \tors{\td,[\cdot,\cdot]_\circ}(D_1,D_2)}{\text{Torsion $\tors{\td,[\cdot,\cdot]}$ on $\A\otimes \D$}}\\
    \underset{\curv{\td,[\cdot,\cdot]_\circ}(D_1,D_2,D_3)}{\text{Curvature $\curv{\td,[\cdot,\cdot]_\circ}$ on $\D$}} ~&\Longrightarrow~ \underset{\substack{\curv{\td,[\cdot,\cdot]}(a_1\otimes D_1,a_2\otimes D_2,a_3\otimes D_3)\\=a_1a_2a_3\otimes \curv{\td,[\cdot,\cdot]_\circ}(D_1,D_2,D_3)}}{\text{Curvature $\curv{\td,[\cdot,\cdot]}$ on $\A\otimes\D$}}
\end{aligned}}
\end{equation*}
\medskip 

The covariant derivative of the torsion $(a_1\otimes D_1) \td \tors{\td,[\cdot,\cdot]}$, along $a_1\otimes D_1\in\A\otimes \D$ is therefore given by:
$$\Big( (a_1\otimes D_1) \td \tors{\td,[\cdot,\cdot]}\Big)(a_2\otimes D_2,a_3\otimes D_3):=a_1a_2a_3\otimes (D_1 \td\tors{\td,[\cdot,\cdot]_\circ})(D_2,D_3)$$
In particular, we obtain that:
\begin{align*}
(a_1\otimes D_1) \td\tors{\td,[\cdot,\cdot]}=0 ~ &\Leftrightarrow ~ D_1 \td\tors{\td,[\cdot,\cdot]_\circ}= 0\\
\curv{\td,[\cdot,\cdot]}(a_1\otimes D_1,a_2\otimes D_2,a_3\otimes D_3)=0 ~&\Leftrightarrow~ \curv{\td,[\cdot,\cdot]_\circ}(D_1,D_2,D_3)=0
\end{align*}
We obtain as a direct corollary of the Theorem \ref{th: geometric post-Lie algebra} that:
\begin{corollary}[of Theorem \ref{th: geometric post-Lie algebra}]
    $(\A\otimes \D,\td,-\tors{\td,[\cdot,\cdot]})$ is a post-Lie algebra.
\end{corollary}

Then we remark that by definition of $\tr$ in \eqref{eq: post-lie product derivation} and of $\td$ in \eqref{eq: extension connections on derivations}, we have for all $a\in\A$ and $D\in \D$, that $(a\otimes D) \tr \cdot \in\Der (\A\otimes \D,\td)$, which is the compatibility condition \eqref{eq: compatibility derivation}, which gives us the following corollary:
\begin{corollary}[of Proposition \ref{prop: pre-Lie deformation}]\label{coro: pre-Lie deformation for post-Lie alg of der.}
    If $\td$ is a pre-Lie product on $\D$, such that $[\cdot,\cdot]_{\diamond}=[\cdot,\cdot]_\circ$, then $(\A\otimes \D,\tr+\td)$ is a pre-Lie algebra.
\end{corollary}

\medskip

We give below an adaptation of the Theorem \ref{th: deformation post-Lie} to our present setting where $L:=\A\otimes \Der(\A)$, which is a generalisation of Theorem \ref{theo: post-Lie structure from derivations} (\cite{jacques2023post}[Theorem 3.1.]) :

\begin{theorem}\label{th: post-Lie conditions on A tens Der(A)}
    Let $\D\subset Der(\A)$ a subspace of $Der(\A)$ stable by the commutator of the composition product $[\cdot,\cdot]_\circ$ and let $\td:\D\otimes \D\to\D$ be a bilinear operation. Using the notations of Theorem \ref{th: deformation post-Lie}, that is to say:
\begin{align}
    (a_1\otimes D_1)\btr (a_2 \otimes D_2)&:=a_1 D_1(a_2)\otimes D_2 + a_1a_2\otimes (D_1 \td D_2),\label{eq: post-lie product geometric derivation alg}\\
    \lb a_1\otimes D_1, a_2 \otimes D_2\rb&:=a_1a_2\otimes \Big([D_1,D_2]_\circ-[D_1,D_2]_{\diamond}\Big),\label{eq: post-lie bracket geometric derivation alg}
\end{align}
we have that if $\curv{\td,[\cdot,\cdot]_\circ}\equiv 0$ and $\td\tors{\td,[\cdot,\cdot]_\circ}\equiv 0$, then $(\A\otimes \D,\btr,\lb\cdot,\cdot\rb)$ is a post-Lie algebra.
\end{theorem}
\begin{proof} This is a simple application of the deformation's Theorem \ref{th: deformation post-Lie}, for which the compatibility condition \eqref{eq: compatibility derivation} given here as \[x\tr\cdot \in \Der(\A\otimes \D,\td)\qquad \forall x\in \A\otimes \D,\]
is easily satisfied in our context. Indeed for all $a_1,a_2,a_3\in\A$ and all $D_1,D_2,D_3\in \Der(\A)$, we have by associativity and commutativity of $(\A,\cdot)$ the following equalities:
\begin{align*}
    &(a_1\otimes D_1) \tr \Big((a_2\otimes D_2)\td(a_3\otimes D_3)\Big)\\
    &=(a_1\otimes D_1) \tr (a_2a_3\otimes D_2\td D_3)\\
    &=a_1 D_1(a_2) a_3 \otimes D_2\td D_3 + a_1a_2D_1(a_3)\otimes D_2\td D_3\\
    &=\Big((a_1\otimes D_1) \tr (a_2\otimes D_2)\Big) \td (a_3\otimes D_3) + (a_2\otimes D_2)\td \Big((a_1\otimes D_1) \tr (a_3\otimes D_3)\Big)
\end{align*}
\end{proof}

\begin{remark}
    The structure $(\A\otimes \D,\btr,\lb\cdot,\cdot\rb)$ defined above is a generalization of the post-Lie algebra $(\A\otimes \D,\tr,[\cdot,\cdot])$ given in Theorem \ref{theo: post-Lie structure from derivations} from \cite{jacques2023post}, in the sense that the two coincide when $\td\equiv 0$.
\end{remark}

In the particular case if $\D\subset Der(\A)$ is a subspace of commuting derivations for the composition product, that is to say if 
\[D_1\circ D_2=D_2\circ D_1,\qquad ~ D_1,D_2\in \D,\]
then the Lie bracket $[\cdot,\cdot]$ on $\A\otimes \D$ is null and we obtain the following corollary:

\begin{corollary}\label{coro: post-Lie commuting derivations}
    Let us take the same hypothesis as in Theorem \ref{th: post-Lie conditions on A tens Der(A)}. If $\D\subset Der(\A)$ is a subspace of commuting derivations, then  $(\A\otimes \D,\btr,-\tors{\td})$ is a post-Lie algebra, and if moreover $\td \equiv 0$, then $(\A\otimes \D,\btr=\tr)$ is a pre-Lie algebra.
\end{corollary}

\subsection{Representation of the enveloping algebras}\label{subsection: Representation of the enveloping algebras}
Throughout this subsection, we consider $(L,\tr,[\cdot,\cdot])$ as being a sub post-Lie algebra of the canonical post-Lie algebra on $\A\otimes \Der(\A)$ defined in Theorem \ref{theo: post-Lie structure from derivations} and we consider a deformed post-Lie structure $(L,\btr,\lb\cdot,\cdot\rb)$ as given in Theorem \ref{th: post-Lie conditions on A tens Der(A)}.

In \cite{jacques2023post} algebra representations of $(\env_{[\cdot,\cdot]}(L),\star_{\btr})$ and $(\env_{\llbracket \cdot,\cdot \rrbracket}(L),\mathsf{conc})$ on $\A$, that is to say an algebra morphism with values in the space of endomorphisms $\End(\A)$ endowed with the composition product $\circ$, have been given. We are aiming at defining a representation of $(\env_{\lb\cdot,\cdot\rb}(L),\star_{\btr})$ on $\A$:

Consider the linear map $\rho:\A\otimes \Der(\A)\to \Der(\A)$ given by 
\begin{equation}\label{eq:rho}
\rho(a\otimes D)=a\cdot D,
\end{equation} 
where $a\cdot D$ denotes the element of $\End(\A)$ defined by:
\begin{equation}\label{eq:aD}
a\cdot D:\A\to\A, \qquad a\cdot D(b):=aD(b). 
\end{equation}
We have seen that $(\Der(\A),[\cdot,\cdot]_\circ)$ is a sub-Lie algebra of $(\End(\A),[\cdot,\cdot]_\circ)$, and that $(L,\llbracket\cdot,\cdot\rrbracket)$ is a
Lie algebra since $L\subseteq \A\otimes\Der(\A)$ is post-Lie. The relation between these two Lie algebras is explained by the following:
\begin{lemma}[\cite{jacques2023post} Lemma 3.9.]
The map $\rho:(L,\llbracket\cdot,\cdot\rrbracket)\to (\Der(\A),[\cdot,\cdot]_\circ)$ is a morphism of Lie algebras.
\end{lemma}

In \cite{jacques2023post} algebra representations of $(\env_{[\cdot,\cdot]}(L),\star_{\btr})$ and $(\env_{\llbracket \cdot,\cdot \rrbracket}(L),\mathsf{conc})$ on $\A$, that is to say an algebra morphism with values in the space of endomorphisms $\End(\A)$ endowed with the composition product $\circ$, have been given. We recall below the two representation morphisms $\rhohat$ and $\rhobar$:

\begin{theorem}\label{theo: representation} The following linear maps $\rhohat$ and $\rhobar$ are morphism of algebras:
\begin{equation}\label{eq:rhohat}
\writefun{\rhohat}{ \left(\env_{\llbracket \cdot,\cdot \rrbracket}(L),\mathsf{conc}\right)}{\left(\End(\A),\circ\right)}{(a_1\otimes D_1)\cdots(a_n\otimes D_n)}{(a_1\cdot D_1)\circ\cdots\circ(a_n\cdot D_n)}
\end{equation}
\begin{equation}\label{eq:rhobar}
\writefun{\rhobar_{\tr}}{\left(\env_{[\cdot,\cdot]}(L),\star_{\tr}\right)}{\left(\End(\A),\circ\right)}{(a_1\otimes D_1)\cdots(a_n\otimes D_n)}{a_1\cdots a_n\cdot (D_1\circ \ldots \circ D_n)}
\end{equation}
\end{theorem}

where the universal property of $\env_{\llbracket \cdot,\cdot \rrbracket}(L)$, 
the map $\rhohat$ is the unique extension of $\rho$, into a morphism of associative algebras.

We also define by composition a representation morphism \[\rhobar_{\btr}:\left(\env_{\lb\cdot,\cdot\rb}(L),\star_{\btr}\right)\to\End(\A),\] by setting:
\begin{equation*}
\rhobar_{\btr}:=\rhobar_{\tr}\circ \Phi_{\tr} \circ \Phi_{\btr}^{-1}.
\end{equation*}

To summarise, let us complete the commutative diagramm \eqref{diag: envl alg} into a commutative diagram of associative algebras, by adding the representation maps:
\[
\begin{tikzcd}[scale cd=1,sep=large]
    & \lp \env_{\llbracket \cdot,\cdot \rrbracket}(L),\mathsf{conc}\ \rp \arrow[dd,dashed,"\rhohat", pos=0.3] \arrow["\sim"']{dr}{\Phi_{\btr}} \\
    \left(\env_{[\cdot,\cdot]}(L),\star_{\tr}\right) \arrow["\sim"']{ur}{\Phi_{\tr}^{-1}}\arrow[rr,"\sim", pos=0.4, crossing over] \arrow[dashed,swap]{dr}{\rhobar_{\tr}} && \left(\env_{\lb\cdot,\cdot\rb}(L),\star_{\btr}\right) \arrow[dashed]{dl}{\rhobar_{\btr}} \\
     & (\End(\A),\circ)
  \end{tikzcd}
\]

\begin{remark}\label{rem: equality rho}
    We remark that for all basis element $a\otimes D \in \B_L$:
    \begin{equation}
        \rho(a\otimes D)=\rhobar_{\tr}(a\otimes D)=\rhobar_{\btr}(a\otimes D).
    \end{equation}
\end{remark}

We are now aiming at characterising the morphism $\rhobar_{\btr}$ using a map \[\Psi_{\diamond}:\tens(\Der(\A))\to\End(\A)\] defined for all $D_0,D_1,\ldots, D_n\in \Der(\A)$, by induction on the length of the words as:
\begin{equation}\label{eq: Psi_diamond}
\begin{aligned}
    \Psi_{\diamond}[D_0 D_1\cdots D_n]:=&D_0\circ \Psi_{\diamond}[D_1\cdots D_n]\\
    &- \sum_{i=1}^n \Psi_{\diamond}[D_1\cdots(D_0\td D_i)\cdots D_n]
\end{aligned}
\end{equation}

We remark that if $\td\equiv0$, we obtain the iterated composition of endomorphisms:
\[\Psi_{\0}[D_1\cdots D_n]=D_1\circ \ldots \circ D_n\]
which implies that the following Theorem is a generalization of \cite{jacques2023post}[Theorem 3.11.]:
\begin{theorem}\label{theo: deformed representation} The linear map $\rhobar_{\btr}$
admits the following explicit expression:
\begin{equation}\label{eq: representation deformed algebra derivation}
    \rhobar_{\btr}\Big((a_1\otimes D_1)\cdots (a_n\otimes D_n)\Big)= a_1\cdots a_n\cdot \Psi_{\diamond}[D_1\cdots D_n].
\end{equation}
By the algebra morphism property for $\rhobar_{\btr}:\left(\env_{[\cdot,\cdot]}(L),\star_{\btr}\right)\to(\End(\A),\circ)$, we also have
\begin{equation}\label{eq: representation algebra derivation2}
\rhobar_{\btr}\Big((a_1\otimes D_1)\star_{\btr} \cdots\, \star_{\btr}\,(a_n\otimes D_n) \Big)=(a_1\cdot D_1)\circ\cdots\circ(a_n\cdot D_n).
\end{equation}
\end{theorem}
\begin{proof} This proof is an adaptation of the proof of \cite{jacques2023post}[Theorem 3.11.].
To prove \eqref{eq: representation deformed algebra derivation}, we proceed by induction on $n$: for
$n=1$ the claim follows from the definition \eqref{eq:rho} of $\rho$. Let us suppose now that  \eqref{eq: representation deformed algebra derivation} is proved for $n\geq 1$; let us set for ease of notation $x_i:=a_i\otimes D_i\in L$, for all $i=0,\ldots,n$; then 
by \eqref{eq:atrbbb}:
\[
x_0\star_{\btr} (x_1\cdots x_n) = x_0 \cdots x_n + \sum_{i=1}^n x_1 \cdots (x_0\btr x_i) \cdots x_n.
\]
By the definition of $\btr$ we have $x_0\btr x_i=a_0D_0(a_i)\otimes D_i+a_0a_iD_0\td D_i$. 
Applying $\rhobar_{\btr}$ to both sides of the above equality, and using linearity, we obtain by the induction hypothesis:
\begin{align*}
    \rhobar_{\btr}\Big(x_0\star_{\btr} (x_1\cdots x_n)\Big)&=\rhobar_{\btr}(x_0\cdots x_n)\\
    &+\sum_{i=1}^n a_0\cdots D_0(a_i) \cdots a_n \cdot \Psi_{\diamond}[D_1\cdots D_n] \\
    &+ \sum_{i=1}^n a_0\cdots a_n \cdot \Psi_{\diamond}[D_1\cdots(D_0\td D_i)\cdots D_n]
\end{align*}

On the other hand, by the morphism property and the induction hypothesis:
\[
\begin{split}
&\rhobar_{\btr}\Big(x_0\star_{\btr} (x_1\cdots x_n)\Big) =\rhobar_{\btr}(x_0)\circ \rhobar_{\btr}(x_1\cdots x_n)
\\ & = (a_0 \cdot D_0)\circ \left(a_1\cdots a_n\cdot\Psi_{\diamond}[D_1\cdots D_n]\right)
\\ & = \sum_{i=1}^n a_0\cdots D_0(a_i) \cdots a_n \cdot\Psi_{\diamond}[D_1\cdots D_n]+a_0\cdots a_n\cdot(D_0\circ \Psi_{\diamond}[D_1\cdots D_n])
\end{split}
\]
Therefore we obtain as required by definition of $\Psi_{\diamond}$:
\begin{align*}
    &\rhobar_{\btr}(x_0 \cdots x_n)=\rhobar_{\btr}\Big((a_0\otimes D_0)\cdots (a_n\otimes D_n)\Big)\\
    &=a_0\cdots a_n \cdot\left(D_0\circ \Psi_{\diamond}[D_1\cdots D_n]-\sum_{i=1}^n \Psi_{\diamond}[D_1\cdots (D_0\td D_i)\cdots D_n]\right)\\
    &=a_0\cdots a_n\cdot\Psi_{\diamond}[D_0 D_1\cdots D_n]
\end{align*}
and the proof is complete.
\end{proof}
In particular, we remark that by the Definition of $\td$ on $L$ \eqref{eq: extension connections on derivations} and by commutativity of $(\A,\cdot)$ we have:
\begin{multline*}
    \rhobar_{\btr}\Big((a_1\otimes D_1)\cdots((a_0\otimes D_0)\td (a_i\otimes D_i))\cdots (a_n\otimes D_n)\Big)\\
    = a_0 \cdots a_n\cdot \Psi_{\diamond}[D_1 \cdots (D_0\td D_i)\ldots D_n].
\end{multline*}
Therefore, we have that:
\begin{multline}\label{eq: recurrence deformation rhobar}
        \rhobar_{\btr}(u)(b)=a_1\cdots a_n\cdot \rho(a_0\otimes D_0)\Big(\Psi_\diamond[D_1\cdots D_n](b)\Big)\\
        -\sum_{i=1}^n \rhobar_{\btr}\Big((a_1\otimes D_1)\cdots((a_0\otimes D_0)\td (a_i\otimes D_i))\cdots (a_n\otimes D_n)\Big)
    \end{multline}

\begin{prop}\label{prop: multiplicativity property Psi}
    Given $b_1,b_2\in \A$ and $U=D_1\cdots D_n\in \tens(\Der(\A))$, using Sweedler's notation \eqref{eq: Sweedler}, we have the following equalities in $\A$:
    \begin{equation}
        \begin{aligned}
        \Psi_{\diamond}[U](b_1 b_2)&=m_\A(\Psi_{\diamond}\otimes \Psi_{\diamond})(\Delta_\ast U)(b_1\otimes b_2)\\
        &=\sum_{\Delta_\ast(U)} \Psi_{\diamond}\left[U^{(1)}\right](b_1)\cdot\Psi_{\diamond}\left[U^{(2)}\right](b_2)\\
        &=\sum_{I\sqcup J=\{1,\ldots,n\}} \Psi_{\diamond}\left[\prod_{i\in I}D_i\right](b_1)\cdot\Psi_{\diamond}\left[\prod_{j\in J}D_j\right](b_2)
    \end{aligned}
    \end{equation}
    \end{prop}
    \begin{proof} We prove the formula by induction on the length $n\geq 1$ of $U=D_1\cdots D_n$. For $n=1$, there is nothing to prove since $\Psi_\diamond[D]=D$ for all $D\in \Der(\A)$. We suppose that the formula is proven for $U=D_1\cdots D_n$ up to a certain order $n\geq 1$, and we prove it for $D_0U=D_0D_1\cdots D_n$, with $D_0\in\Der(\A)$. We have by definition:
    \[\Psi_\diamond[D_0 U](b_1b_2)=D_0\circ \Psi_\diamond[U](b_1b_2)-\Psi_\diamond[D_0 \td U](b_1b_2)\]
    For the first term of the right side of the last equality, by induction hypothesis, linearity and Leibniz rule of $D_0$ on $\A$:
    \begin{multline}
        D_0\circ \Psi_\diamond[U](b_1b_2)=\sum_{\Delta_\ast(U)} \left(D_0\circ\Psi_{\diamond}\left[U^{(1)}\right](b_1)\right)\cdot\Psi_{\diamond}\left[U^{(2)}\right](b_2)\\
        + \Psi_{\diamond}\left[U^{(1)}\right](b_1)\cdot \left(D_0\circ\Psi_{\diamond}\left[U^{(2)}\right](b_2)\right)
    \end{multline}
    For the second term, since
    \[\Delta_\ast (D_0\td U)=\sum_{\Delta_\ast(U)}(D_0\td U^{(1)})\otimes U^{(2)}+ U^{(1)}\otimes (D_0\td U^{(2)}),\]
    we have by induction hypothesis:
    \begin{multline*}
        \Psi_\diamond[D_0 \td U](b_1b_2)=\sum_{\Delta_\ast(U)}\Psi_\diamond[D_0\td U^{(1)}](b_1)\cdot \Psi_\diamond[U^{(2)}](b_2)\\ + \Psi_\diamond[U^{(1)}](b_1)\cdot \Psi_\diamond[D_0\td U^{(2)}](b_2)
    \end{multline*}
    Therefore, after factorisation:
    \begin{align*}
        &\Psi_\diamond[D_0 U](b_1b_2)\\
        &=\sum_{\Delta_\ast(U)} \Psi_{\diamond}\left[D_0 U^{(1)}\right](b_1)\cdot\Psi_{\diamond}\left[U^{(2)}\right](b_2) + \Psi_{\diamond}\left[U^{(1)}\right](b_1)\cdot \Psi_{\diamond}\left[D_0 U^{(2)}\right](b_2) \\
        &=m_\A(\Psi_{\diamond}\otimes \Psi_{\diamond})\Delta_\ast (D_0U)(b_1\otimes b_2)
    \end{align*}
    \end{proof}
We obtain as corollary an analog of \cite[Proposition 3.14]{jacques2023post}:
\begin{corollary}\label{prop: multiplicativity property representation}
    Given $b_1,b_2\in \A$ and $u=(a_1\otimes D_1)\cdots (a_n\otimes D_n)\in \env_{[\cdot,\cdot]}(\A\otimes\Der(\A))$:
    \begin{equation}
        \begin{aligned}
        \rhobar_{\btr}(u)(b_1b_2)&=m_\A(\rhobar_{\btr}\otimes \rhobar_{\btr})(\Delta_\ast u)(b_1\otimes b_2)\\
        &=\sum_{\Delta_\ast(u)} \rhobar_{\btr}(u^{(1)})(b_1)\rhobar_{\btr}(u^{(2)})(b_2)\\
        &=\sum_{I\sqcup J=\{1,\ldots,n\}}\rhobar_{\btr}\left(\prod_{i\in I}(a_i\otimes D_i)\right)(b_1)\, \rhobar_{\btr}\left(\prod_{j\in J}(a_j\otimes D_j)\right)(b_2)
        \end{aligned}
    \end{equation}
\end{corollary}
\begin{proof} It is a simple factorisation of the terms $a_1\cdots a_n$, ideed on one side:
    \[\rhobar_{\btr}(u)(b_1b_2)=a_1\cdots a_n\Psi_{\diamond}[D_1\cdots D_n](b_1b_2)\]
    on the other side:
    \begin{multline*}
        \sum_{I\sqcup J=\{1,\ldots,n\}}\rhobar_{\btr}\left(\prod_{i\in I}(a_i\otimes D_i)\right)(b_1)\, \rhobar_{\btr}\left(\prod_{j\in J}(a_j\otimes D_j)\right)(b_2)\\
        = a_1\cdots a_n \sum_{I\sqcup J=\{1,\ldots,n\}} \Psi_{\diamond}\left[\prod_{i\in I}D_i\right](b_1)\cdot\Psi_{\diamond}\left[\prod_{j\in J}D_j\right](b_2)
    \end{multline*}
    We conclude using Proposition \ref{prop: multiplicativity property Psi}.
\end{proof}

\medskip
\section{Application to regularity structures: a pre-Lie algebra structure.}\label{sec: regularity structures}
\subsection{Context.}
We want here to give an application of the results of the previous subsection to new tree-free approach to regularity structures developed in \cite{LOT}, using the point of view of \cite{jacques2023post}. We first remind briefly some results and we explain how it fits in our present setting:\\
We note $\N=\{0,1,\ldots\}$ and given an integer $d\geq 1$, we use the following notations: 
\[
\N^d_*:=\N^d\setminus\{\0\}, \qquad \0:=(0,\ldots,0)\in\N^d,
\]
and we consider the polynomial algebra:
\begin{equation*}
\A:=\R[\z_k,\z_\n]_{k\in\N,\n\in\N^d_*}
\end{equation*}
where $\{\z_k,\z_\n: k\in\N,\n\in\N^d_*\}$ are the indeterminates, which is endowed with the free commutative and associative product, and for which we denote by $\1\in\A$ the unit.

To write the monomials, we adopt as in \cite{jacques2023post} the \textit{multi-index notation}: we define $\M$ as the set of compactly supported $\gamma:\N\sqcup\N^d_*\to\N$, namely $\gamma_i\ne0$ only for finitely many $i\in\N\sqcup\N^d_*$. Elements of $\M$ are called \textit{multi-indices}.  

A monomial of the polynomial algebra $\A$ is therefore given by:
$$\z^\gamma:=\prod_{i\in\N\sqcup\N^d_*}\z_i^{\gamma_i}, \quad \gamma\in\M,
\qquad \z^0=\1.$$

Note that the addition is defined on $\M$ as: if $\gamma^1,\gamma^2\in\M$ then 
\begin{equation}\label{eq:sumga}
\gamma_i:=\gamma^1(i)+\gamma^2(i), \qquad i\in\N\sqcup\N^d_*,
\end{equation}
defines a new element in $\M$. Then the sum in $\M$ defined in \eqref{eq:sumga} allows to describe the product in $\A$
\[
\z^{\gamma}\z^{\gamma'}=\z^{\gamma+\gamma'}, \qquad \gamma,\gamma'\in\M.
\]

As in \cite{jacques2023post}, we consider a space $\D$ of derivations on $\A$ that is freely generated by a family of derivations $\B_\D\subset \D$ (which is therefore a linear basis of $\D$), which can be split into two subfamilies (see \cite[(3.9) and (3.12)]{LOT}):
\begin{itemize}
    \item The \textit{tilt} derivations $\{D^{(\n)}\}_{\n\in\N^d}$, defined by:
    \begin{equation}\label{eq: tilt}
        D^{(\0)}:=\sum_{k\ge 0}(k+1)\z_{k+1}\partial_{\z_k}\qquad \text{and}\qquad D^{(\n)}:=\partial_{z_\n},~ \text{for}~ \n\in\N^d_*.
    \end{equation}
    \item The \textit{shift} derivations $\{\partial_i\}_{i\in\{1,\ldots,d\}}$, defined by:
    \begin{equation}\label{eq: shift}
        \partial_i:=\sum_{\n\in\N^d}(n_i+1)\z_{\n+\e_i}D^{(\n)}
    \end{equation}
    where $\e_1=(1,0,\ldots,0)$, $\e_2=(0,1,\ldots,0)$, etc.
\end{itemize}
Denoting by $e_k\in\M$ the multi-index $e_k(i)=\ind_{(i=k)}$ for $k\in\N$ and $i\in\N\sqcup\N^d_*$, and
similarly $e_\n\in\M$ for $\n\in\N^d_*$, we can compute the action of the derivations on $\z^\gamma\in\A$:
\begin{equation}\label{eq:D0}
    D^{(\0)}\z^\gamma=\sum_{k\geq 0}(k+1)\gamma_k\z^{\gamma+e_{k+1}-e_k}, \qquad
    D^{(\n)}\z^\gamma=\gamma_\n\,\z^{\gamma-e_\n},\quad\n\in\N^d_*
\end{equation}
and  for all $i\in\{1,\ldots,d\}$:
\begin{equation}\label{eq:partial_i}
    \partial_i\z^\gamma=\sum_{k\geq 0}(k+1)\gamma_k\z^{\gamma+e_{k+1}-e_k+e_{\e_i}}+\sum_{\n\in\N^d_*}(n_i+1)\gamma_\n\z^{\gamma-e_\n+e_{\n+\e_i}}.
\end{equation}
Therefore the linear basis for $\D\subset \Der(\A)$ is
\begin{equation}
    \B_\D:=\{\partial_i\}_{i\in\{1,\ldots,d\}}\cup \{D^{(\n)}\}_{\n\in\N^d}.
\end{equation}
Then, our derivations space of interest is given as:
$$\D:=\R\cdot \B_\D$$

For the computation of the composition product in $(\Hom(\A),\circ)$ between derivations of this family, we refer to \cite[Section 4.2]{jacques2023post}. The Lie bracket $[\cdot,\cdot]_\circ$ on $\Der(\A)$, given as the commutator for the composition product, is given for all $\n,\n'\in\N^d$, with $\n=(n_1,\ldots,n_d)$ and all $i,j\in\{1,\ldots,d\}$ by:
\begin{align}
    [D^{(\n)},D^{(\n')}]_\circ&=0,\\
    [\partial_i,\partial_j]_\circ&=0,\\
    [\partial_i, D^{(\n)}]_\circ&= -n_iD^{(\n-\e_i)}\in \D,
\end{align}
with the understanding that $\partial_i\td D^{(\n)}$ vanishes if $n_i=0$, in particular $[\partial_i, D^{(\n)}]_\circ=0$. Those equalities show that $(\D,[\cdot,\cdot]_\circ)$ is a Lie subalgebra of $(\Der(\A),[\cdot,\cdot]_\circ)$. Note that in particular, for all $i\in\{1,\ldots,d\}$:
\[[\partial_i,D^{(\0)}]_\circ=0.\]

The stability of $\D$ under $[\cdot,\cdot]_\circ$ along with Theorem \ref{theo: post-Lie structure from derivations}, implies as a direct corollary that:
\begin{prop}
    $(\A\otimes \D,\tr,[\cdot,\cdot])$ is a Post-Lie algebra, for the the post-Lie structure $(\tr,[\cdot,\cdot])$ has been defined in Theorem \ref{theo: post-Lie structure from derivations}.
\end{prop}

Then as in \cite[Section 4.2]{jacques2023post}, denoting $\1$ the unit element in $\A$, we define the space $L_0$ as the subspace of $\A\otimes \Der(\A)$ generated by the basis elements $\{\z^\gamma \otimes D^{(\n)}\}_{\n\in\N^d,\gamma\in\M}$ and $ \{\1\otimes \partial_i\}_{i\in\{1,\ldots,d\}}$, namely:
\begin{equation}\label{eq: post-Lie algebra L_0}
L_0:=\R\{\1\otimes\partial_i\}_{i\in\{1,\ldots,d\}} \oplus \R\left\{\z^\gamma \otimes D^{(\n)}\right\}_{\gamma\in\M,\n\in\N^d},
\end{equation}

\begin{theorem}[\cite{jacques2023post} Theorem 4.1]\label{thm:postLOT} The space $L_0$ is a sub-post-Lie algebra of $\A\otimes \D$, for the canonical post-Lie algebra structure $(\tr,[\cdot,\cdot])$ given in Theorem \ref{theo: post-Lie structure from derivations}.
\end{theorem}

\subsection{A pre-Lie algebra structure.}\label{subsec: Pre-Lie regularity structures}
Now we are aiming at applying the Proposition \ref{prop: pre-Lie deformation} to our present post-Lie structure $(L_0,\tr,[\cdot,\cdot])$. We are looking for a bilinear operation $\td\in \Hom(\D^{\otimes 2},\D)$ such that $[\cdot,\cdot]_{\diamond}=[\cdot,\cdot]_\circ$ and such that $(\D,\td)$ is a pre-Lie algebra.

We follow the procedure given in remark \ref{rem: construction pre-Lie product}, considering the subset of $(\B_\D)^2$ of couples of derivations $(D,D')$ such that $[D,D']_\circ\neq 0$ and fixing a strict total order $<$ on it, we define: 
\begin{equation}
        D\td D' := \begin{cases}
        [D,D']_\circ\quad\text{if}\quad ([D,D']_\circ\neq 0) \wedge (D<D')\\
        0\qquad\qquad\text{else}
    \end{cases}
    \end{equation}
As seen above, those couples are the couples of the type $(D^{(\n)},\partial_i)$ (and conversely $(\partial_i,D^{(\n)})$), with $\n\in\N^d_*,~i\in\{1,\ldots,d\}$, and we make the following choice of ordering: 
$$\partial_i<D^{(\n)},\qquad \forall~\n\in\N^d_*,~\forall~i\in\{1,\ldots,d\}$$
Therefore, we obtain the following multiplication table for $\td$ on the basis elements $\B_\D$ of the derivations space $\D$, for all $\n,\n'\in\N^d$, with $\n=(n_1,\ldots,n_d)$ and all $i,j\in\{1,\ldots,d\}$:
\begin{equation}\label{eq: multiplication table}
\begin{aligned}
    D^{(\n)}\td D^{(\n')}&=0,\qquad &\partial_i\td \partial_j&=0 \\
    D^{(\n)} \td \partial_i&=0, &\partial_i\td D^{(\n)}&=-n_iD^{(\n-\e_i)} 
\end{aligned}
\end{equation}
With the understanding that $\partial_i\td D^{(\n)}$ vanishes if $n_i=0$, in particular, we have that for all $i\in\{1,\ldots,d\}$:
\[\partial_i\td D^{(\0)}=0\]
Let us now prove that $\td$ is a pre-Lie product, that is to say, we prove that
$\ass_{\td}(D_1,D_2,D_3)$ is symmetrical in $(D_1,D_2)$ for all $D_1,D_2,D_3\in\D$:\\
Fix $\n,\n',\n''\in\N^d_*$ and $i,j,k\in\{1,\ldots,d\}$. 
From the null cases in the multiplication table of $\td$ above, we remark that:
$$\ass_{\td}(D^{(\n)},D^{(\n')},D^{(\n'')})=0,\quad \text{and\quad }\ass_{\td}(D_1,D_2,\partial_i)=0,\quad \forall D_1,D_2\in \D$$ 
which drastically reduces the number of cases to be treated.
Then, denoting $(n_1,\ldots,n_d)$, and by writing down the associator $\ass_{\td}$, we have that:
\begin{align*}
    \ass_{\td}(D^{(\n)}, \partial_i , D^{(\n')})&=D^{(\n)} \td (\partial_i \td D^{(\n')}) - (D^{(\n)} \td \partial_i) \td D^{(\n')}\\
    &=-n_i' (D^{(\n)} \td D^{(\n'-\e_i)})-0\\
    &=0,
\end{align*}
since $D^{(\n)} \td \partial_i=D^{(\n)} \td D^{(\n'-\e_i)}=0$.\\
We also have that:
\begin{align*}
    \ass_{\td}(\partial_i, D^{(\n)} , D^{(\n')})&=\partial_i \td ( D^{(\n)}\td D^{(\n')}) - (\partial_i \td D^{(\n)}) \td D^{(\n')}\\
    &=0+n_i(D^{(\n-\e_i)} \td D^{(\n')})\\
    &=0,
\end{align*}
since $D^{(\n)}\td D^{(\n')}=D^{(\n-\e_i)} \td D^{(\n')}=0$.\\
Finally, since $\partial_i \td \partial_j=0$:
\begin{align*}
    \ass_{\td}(\partial_i, \partial_j , D^{(\n)})&=\partial_i \td ( \partial_j\td D^{(\n)}) - (\partial_i \td \partial_j) \td D^{(\n)}\\
    &=\partial_i \td ( \partial_j\td D^{(\n)})-0\\
    &=\begin{dcases}
    n_i (n_i-1)D^{(\n-2\e_i)} \quad &\text{if}\quad i=j\\
    n_i n_j D^{(\n-\e_i-\e_j)}  &\text{if}\quad i\neq j
    \end{dcases}
\end{align*}
which is symmetric in $(\partial_i ,\partial_j)$.\\

We can extend $\td$ from $\D\subset \Der(\A)$ to $\A\otimes \D\subset \A\otimes \Der(\A)$ as a bilinear operation $\td\in\Hom((\A\otimes \D)^{\otimes 2},\A\otimes \D)$ as made before in \eqref{eq: extension connections on derivations}. We derive from \eqref{eq: multiplication table} the following multiplication table of $\td$ on $L_0\subset \A\otimes \Der(\A)$, given for all $\gamma,\gamma'\in\M$, $\n,\n'\in\N^d_*$, and $i,j\in\{1,\ldots,d\}$ by:
\begin{equation}\label{eq: multiplication table connection}
\begin{aligned}
    (\1\otimes \partial_i)&\td (\1\otimes \partial_j)& &=  \1\otimes (\partial_i\td \partial_j)=\1\otimes 0=0,\\
    (\z^\gamma \otimes D^{(\n)}) &\td  (\1\otimes \partial_i)&  &=\z^\gamma \otimes (D^{(\n)} \td \partial_i)=\z^\gamma \otimes 0= 0, \\
    (\z^{\gamma'}\otimes D^{(\n')})&\td  (\z^\gamma \otimes D^{(\n)})& &=\z^{\gamma+\gamma'}\otimes (D^{(\n')}\td D^{(\n)})=\z^{\gamma+\gamma'}\otimes0=0,\\
    (\1\otimes \partial_i) &\td  (\z^\gamma \otimes D^{(\n)})& &= \z^\gamma\otimes (\partial_i\td D^{(\n)})=-n_i(\z^\gamma\otimes  D^{(\n-\e_i)}) 
    \end{aligned}
\end{equation}
Futhermore, we observe that $L_0$ is stable by $\td$, which brings us to the following corollary:

\begin{theorem}\label{theo: pre-Lie deformation RG} $(\A\otimes \D,\tr+\td)$ is a pre-Lie algebra, and $(L_0,\tr+\td)$ is a sub-pre-Lie algebra. These are respectively the $\td$-gpL deformation of the post-Lie algebras $(\A\otimes \D, \tr,[\cdot,\cdot])$, and of its sub-post-Lie algebra $(L_0,\tr,[\cdot,\cdot])$.
\end{theorem}
\begin{proof}
    We have proved above that $(\D,\td)$ is a post-Lie algebra and that $[\cdot,\cdot]_{\diamond}=[\cdot,\cdot]_\circ$, which are the hypothesis of the Corollary \ref{coro: pre-Lie deformation for post-Lie alg of der.}, which implies that $(\A\otimes \D,\tr+\td)$ is a pre-Lie algebra. The multiplication table \eqref{eq: multiplication table connection} shows that $L_0$ is stable by $\td$ in $\A\otimes \D$, and as it is also stable by $\tr$, we deduce that it is stable by $\tr+\td$, which concludes the proof.
\end{proof}
Denoting again $\btr:=\tr+\td$, we can derive easily its multiplication table on $\B_{L_0}$:
\begin{align}
    (\1\otimes \partial_i)&\btr (\1\otimes \partial_j)& &=0, \label{eq: partial btr partial}\\
    (\z^\gamma \otimes D^{(\n)}) &\btr (\1\otimes \partial_i)& &=0, \label{eq: der btr partial}\\
    (\z^{\gamma'}\otimes D^{(\n')})&\btr (\z^\gamma \otimes D^{(\n)})& &=\z^{\gamma'} D^{(\n')}\z^{\gamma}\otimes D^{(\n)}, \label{eq: der btr der}\\
    (\1\otimes \partial_i) &\btr (\z^\gamma \otimes D^{(\n)})& &=\partial_i\z^\gamma\otimes D^{(\n)}-n_i(\z^\gamma\otimes  D^{(\n-\e_i)})\label{eq: partial btr der}
\end{align} 

We mention the interesting observation that the commutator $[\cdot,\cdot]_{\btr}$ is equal to the commutator $[\cdot,\cdot]_\tr$ of \cite[\S 3.10]{LOT}. However the space $L$ is nons-table under the bilinear operation $\tr$ defined in \cite[\S 3.8]{LOT}.

\subsection{The dual Hopf algebra structure}
Recalling the Guin-Oudom Hopf algebra structure of Theorem \ref{prop: associative product}, we have a Hopf algebra $(\sym(L),\star_{\btr},\Delta_\ast,\ind,\varepsilon)$.
In that section, we aim at constructing the dual structure $(\ast,\Delta_{\star_{\btr}})$ of $(\star_{\btr},\Delta_\ast)$ on $\sym(L)$. 

We first define the product $\ast$ on the symmetric algebra $\sym(L)$ that is dual to the coshuffle coproduct $\Delta_\ast$ and that we will need later to define the structure group. We begin by defining two dual bases:
\begin{itemize}
    \item The monomial basis $\overline{\B}_{\sym(L)}$ of $\sym(L)$ is given as the set of monomials: 
    \begin{equation}\label{eq: monomial basis}
    \Bbar_{\sym(L)}:=\{\mathds 1\}\sqcup\Bigg\{\,x_1^{m_1}\cdots x_k^{m_k}: \, k, m_1,\ldots,m_k\geq 1,\, x_i\in\B_L
    \Bigg\}.
    \end{equation}
    \item The twisted monomial basis is another basis of $\sym(L)$ given by renormalising the monomial by their symmetry factor, that is to say:
    \begin{equation}\label{eq: twisted monomial basis}
    \B_{\sym(L)}:=\{\mathds 1\}\sqcup\Bigg\{\frac1{m_1!\cdots m_k!}\,x_1^{m_1}\cdots x_k^{m_k}: \, k, m_1,\ldots,m_k\geq 1,\, x_i\in\B_L
    \Bigg\}.
    \end{equation}
\end{itemize}
 
In particular, an element of $\Bbar_{\sym(L)}\setminus\{\ind\}$ can be written:
\begin{equation}\label{eq:basis1}
        (\1\otimes\partial_1)^{m_1}\ldots (\1\otimes\partial_d)^{m_d}(\z^{\gamma_1}\otimes D^{(\n_1)} )\ldots (\z^{\gamma_k}\otimes D^{(\n_k)})
    \end{equation}
    where $(m_1,\ldots,m_d)\in \N^d$ and $(\gamma_i, \n_i)\in \M\times \N^d$ for all $i\in\{1,\ldots,k\}$. We have a map $T:\B_{\sym(L)}
\to{\overline\B}_{\sym(L)}$ given by $T(\mathds 1)=\mathds 1$ and
\begin{equation}\label{eq:T}
T\left(\frac1{m_1!\cdots m_k!}x_1^{m_1}\cdots x_k^{m_k}\right)=x_1^{m_1}\cdots x_k^{m_k},
\end{equation}
which has a unique linear extension $T:\envU(L)\to\envU(L)$.
Then we introduce the pairing on $\sym(L)\otimes\sym(L)$ given by the bilinear extension of 
\begin{equation}\label{eq:dualityL}
\B_{\sym(L)}\times{\overline\B}_{\sym(L)}\ni(u,v)\mapsto \la u,v\ra :=\ind_{(Tu=v)}.
\end{equation}

As in \cite[\S 2.5]{jacques2023post}, we consider the polynomial product on $\sym(L)$, which is defined on the monomial basis ${\overline\B}_{\sym(L)}$ for all basis element $x_i\in\B_L$ and $n_i,m_i\in\N$ by:
\begin{equation}\label{eq: polynomial product symmetric alg}
	\prod_{i\in I} x_i^{n_i}\ast\prod_{i\in I} x_i^{m_i}=\prod_{i\in I} x_i^{n_i+m_i}.
\end{equation}
which gives an associative and commutative graded algebra $(\sym(L),\ast,\ind)$ which is dual to the coproduct $(\Delta_\ast,\varepsilon)$ in the sense that for all $u,v,w\in \sym(L)$:
\begin{equation*}
	\la w,u*v\ra =\la \Delta_\ast w,u\otimes v\ra,
\end{equation*}
\
Now, in order to dualise $(\star_{\btr},\ind)$, we use the key finiteness Assumption of \cite[Assumption 2.14.]{jacques2023post}, which we need to prove that it is satisfied, as shown in the next Proposition:
\begin{prop}\label{prop: finiteness btr}
For all $w\in \B_L$ the set $$\{(u,v)\in \B_L\times\B_L: \la u\btr v,w \ra\neq 0\}$$ is finite.
\end{prop}

\begin{proof}
We decompose:
$$\la u\btr v,w \ra=\la u\tr v,w \ra + \la u\td v,w \ra$$
Therefore, we have:
$$\la u\btr v,w \ra\neq 0 \Rightarrow \Big(\la u\tr v,w \ra\neq 0 ~\vee~ \la u\td v,w \ra\neq 0\Big)$$
It has been shown in \cite[Proposition 5.2. combined with Lemma 3.17.]{jacques2023post} that 
$$\{(u,v)\in \B_L\times\B_L: \la u\tr v,w \ra\neq 0\}$$ is finite. Therefore, it remains to prove the same finiteness property for $\td$:

For all $w\in\B_L$, the finiteness of the couples $(u,v)\in \B_L\times\B_L$ such that $\la u\td v,w \ra\neq 0$ is ensured by the multiplication table of $\td$ on $L$, since the only non-vanishing equality in eqref{eq: multiplication table connection} is:
\[(\1\otimes \partial_i) \td  (\z^\gamma \otimes D^{(\n)})=-n_i(\z^\gamma\otimes  D^{(\n-\e_i)}),\quad \n\geq \e_i\]
We obtain that for all $u,v,w\in\B_L$, $\la u\td v,w \ra$ is non-null if necessarily $w$ is of type $\z^\gamma\otimes D^\m$, with $\gamma\in\M$ and $\m\in\N^d$:
\[\la u\td v,w \ra\neq 0 \,\Longrightarrow \, w=\z^\gamma\otimes D^\m\]
Thus we only need to study this particular case and we have that:
$$\la u\td v,\z^\gamma\otimes D^\m\ra\neq 0 \iff \Big(u=\1\otimes \partial_i \Big) \wedge  \Big(v= \z^\gamma \otimes D^{(\m+\e_i)} \Big),\, i\in\{1,\ldots,d\}$$
Therefore we have proven that for all $w\in \B_L$ the set \[\{(u,v)\in \B_L\times\B_L: \la u\td v,w \ra\neq 0\}\] is finite.
\end{proof}

Denoting again $\la\cdot,\cdot\ra$ the pairing \eqref{eq:dualityL}, we define the linear map:
$\Delta_{\star_{\btr}}:\sym(L)\to\sym(L)\otimes\sym(L)$ given for all $v\in\sym(L)$ by:
\begin{equation}
 \Delta_{\star_{\btr}}v:=\sum_{u_1,u_2\in \B_{\sym(L)}} \la u_1\star_{\btr}  u_2,v\ra (Tu_1)\otimes (Tu_2)
\end{equation}

Since we consider a pre-Lie algebra $(L,\btr)$, which is a post-Lie algebra with null Lie bracket $(L,\btr,\0)$, \cite[Assumption 2.14]{jacques2023post} is satisfied and we can apply the \cite[Corollary 2.16]{jacques2023post} to our particular case where $\envU(L)=\sym(L)$ and $\star=\star_{\btr}$, which give us that for any $w\in\Bbar_{\sym(L)}$ the set \[\{(u,v)\in (\B_{\sym(L)})^2:\,\la u\star_{\btr} v,w \ra\neq 0\}\] is finite.

Then \cite[Proposition 2.17]{jacques2023post} gives us that the product $\star_{\btr}$ can be dualised into a coproduct $\Delta_{\star_{\btr}}$ as in the following Proposition:
\begin{prop} $(\sym(L),\ast,\Delta_{\star_{\btr}},{\mathds 1},\varepsilon)$ is a Hopf algebra, where the product $\ast$ has been defined by \eqref{eq: polynomial product symmetric alg} and the coproduct $\Delta_{\star_{\btr}}:\sym(L)\to\sym(L)\otimes\sym(L)$ is given for all $v\in\sym(L)$ by:
\begin{equation}
 \Delta_{\star_{\btr}}v:=\sum_{u_1,u_2\in \B_{\sym(L)}} \la u_1\star_{\btr}  u_2,v\ra (Tu_1)\otimes (Tu_2)
\end{equation}
Moreover the bialgebra structures $(\star_{\btr},\Delta_\ast)$ and $(\ast,\Delta_{\star_{\btr}})$ are dual relatively to the pairing \eqref{eq:dualityL}, in the sense that:
\begin{align*}
    \la w,u*v\ra &= \la \Delta_\ast w,u\otimes v\ra, \qquad \forall u,v,w\in \sym(L)\\
    \la u\star_{\btr} v,w \ra &=\la u \otimes v,\Delta_{\star_{\btr}}w \ra, \qquad \forall u,v,w\in \sym(L),
\end{align*}
\end{prop}

\subsection{Homogeneity}
In \cite{jacques2023post} a post-Lie algebra $(L,\tr,[\cdot,\cdot])$ is used to construct the structure group of semi-linear SPDEs, considering in particular the equation on $\R^d$:
\begin{equation}\label{eq:spde}
    -\Delta u = a(u)\xi
\end{equation}
where for a fixed $\alpha\in\,]0,1[$:
\begin{itemize}
    \item $u$ is a Hölder function in $C^\alpha(\R^d)$.
    \item $\Delta$ denotes the $d$-dimensional Laplacian operator: $\Delta u= \partial_1^2 u +\ldots+ \partial_d^2 u.$
    \item $\xi$ is a distribution in some Besov space $C^{\alpha-2}$
\end{itemize}

The \emph{homogeneity} is the function $|\cdot|:\M\to[\alpha,+\infty)$ defined as follows:
\begin{equation}\label{eq:|}
|\beta|:=\alpha\sum_{k\geq 0}\beta_k+\sum_{\n\ne \0}|\n|\beta_\n.
\end{equation}
where we denote $|\n|:=|(n_1,\ldots,n_d)|=n_1+\ldots+n_d$ for all $\n\in\N^d$. We remark in particular that the homogeneity is an increasing function:
\[\gamma < \beta \Longrightarrow |\gamma|<|\beta|\]
and that for all $\kappa>0$, the set of multiindices of homogeneity $\kappa$
\[\M_{\kappa}:=\{\beta\in \M,~|\beta|\leq \kappa\}\]
is finite.
We recall the definition \eqref{eq: post-Lie algebra L_0} of $L_0$ and we define the subspace $L\subset L_0$ as in \cite[\S 5.1]{jacques2023post}:
\begin{equation}\label{eq: sub-post-Lie algebra L}
L:=\R\{\1\otimes\partial_i\}_{i\in\{1,\ldots,d\}} \oplus \R\left\{\z^\gamma \otimes D^{(\n)}\right\}_{\gamma\in\M,\, \n\in\N^d, |\gamma|>|\n|}.
\end{equation}

\begin{theorem}[\cite{jacques2023post} Theorem 5.1.]\label{theo: post-Lie algebra structure for L}
    The space $L$ is a sub post-Lie algebra of $\A\otimes \Der(\A)$, for the canonical post-Lie algebra structure $(\tr,[\cdot,\cdot])$ given in Theorem \ref{theo: post-Lie structure from derivations}.
\end{theorem}

The following Theorem is the principal result of our present section. It shows that one can avoid the use of post-Lie structure in \cite{jacques2023post} by defining a pre-Lie structure on $L$:

\begin{theorem} Let $\td$ the bilinear operation on $L$ previously defined by \eqref{eq: multiplication table connection}. $(L,\btr)$, with $\btr:=\tr+\td$ is a pre-Lie algebra and there is an isomophism of Hopf algebras:
    \[(\envU(L),\star_\tr,\Delta_\ast)\xrightarrow{\sim}(\sym(L),\star_{\btr},\Delta_\ast)\]
    between the associative Hopf algebras associated to the post-Lie algebra $(L,\tr,[\cdot,\cdot])$ of \cite{jacques2023post} and the one associated to the pre-Lie algebra $(L,\btr)$.
\end{theorem}
\begin{proof}
    We consider the linear basis $\B_L$ of $L$ given by:
    \begin{equation}\label{eq: B_L}
        \B_L:=\{\1\otimes\partial_i\}_{i\in\{1,\ldots,d\}} \cup \left\{\z^\gamma \otimes D^{(\n)}\right\}_{\gamma\in\M,\, \n\in\N^d, |\gamma|>|\n|},
    \end{equation}
    using \cite[Theorem 5.1.]{jacques2023post}, it is sufficient to prove the stability of $L$ by $\td$. In the multiplication table of $\td$ on $L_0$, the only non-null equality is: 
    \[(\1\otimes \partial_i)\td(\z^\gamma\otimes D^{(\n)})=\z^\gamma\otimes (\partial_i\td D^{(\n)})=-n_i(\z^\gamma\otimes  D^{(\n-\e_i)})\]
    for all $\gamma\in \M,~ \n\in \N^d$. Thus if $|\gamma|>|\n|$ we have that $|\gamma|>|\n|>|\n-\e_i|$ since $|\n-\e_i|=|\n|-1$, which proves that $\z^\gamma\otimes  D^{(\n-\e_i)}\in L$.
\end{proof}

We extend the notion of homogeneity by defining it on $\B_\A$. We denote again $|\cdot|:\B_\A\to [0,+\infty[$ the map given by:
\[|\1|=0, \qquad \forall \gamma\in\M:\quad |\z^\gamma|=|\gamma|\]
We set $A:=\alpha\N+\N=\{\alpha i+j: i,j\in\N\}$. By \eqref{eq:|} the homogeneity $|\beta|$ of $\beta\in\M$
takes values in $A$. We define the following grading on $A$:

\[\A=\bigoplus_{\kappa\in A} \A_\kappa, \qquad \A_0=\R\cdot\{\1\},\quad\A_\kappa=\R\cdot \{z^\gamma\}_{\gamma\in\M_\kappa}.\]

The algebra $(\A,\cdot)$ is a graded algebra, that is to say:
\[\A_{\kappa}\cdot\A_{\kappa'}\in \A_{\kappa+\kappa'}\]

We define also an homogeneity on the derivations $|\cdot|:\B_\D\to \Z$ which is given by:
\[|\partial_i|=1,\qquad |D^{(\n)}|=-|\n|\leq 0\]

Recalling the defining equalities \eqref{eq:D0} and \eqref{eq:partial_i}, we have:
$$D^{(\n)}\z^{\gamma}\in\left\{ 
    \begin{array}{ll} \R\{\z^{\gamma+e_{k+1}-e_k}\}_{k\geq 0}\quad \text{if $\n=\0$},
\\ \\ \R\{\z^{\gamma-e_\n}\} \quad \text{if $\n\ne\0$},
    \end{array}
    \right.$$
and
$$\partial_i\z^\gamma\in \R\{\z^{\gamma+e_{k+1}-e_k+e_{\e_i}}\}_{k\geq 0}\oplus \R\{\z^{\gamma-e_\n+e_{\n+\e_i}}\}_{\n\in\N^d_*}$$

Computing the homogeneity using additivity and that $|e_{k+1}-e_k|=0$ and $|e_\n|=|\n|$ (in particular $|e_{\e_i}|=1$ and $|e_{\n+\e_i}|=|\n|+1$), we obtain:
\[
\begin{split}
& |\gamma+e_{k+1}-e_k|=|\gamma|+|e_{k+1}-e_k|=|\gamma|=|\gamma|+|D^{(\0)}|,
\\ & |\gamma-e_{\n}|=|\gamma|-|\n|=|\gamma|+|D^{(\n)}|,
\\ & |\gamma+e_{k+1}-e_k+e_{\e_i}|=|\gamma|+|e_{k+1}-e_k|+|e_{\e_i}|=|\gamma|+1=|\gamma|+ |\partial_i|
\\ & |\gamma-e_\n+e_{\n+\e_i}|=|\gamma|-|\n| + (|\n| +1)=|\gamma|+1=|\gamma|+ |\partial_i|,
\end{split}
\]

Thus, for all $D\in\B_\D$, we have that $D(\A_\kappa)\subset \A_{\kappa + |D|}$ where if $\kappa + |D|<0$ we denote $\A_{\kappa + |D|}:=\{0\}$.
By the multiplication table \eqref{eq: multiplication table} of $\td$, we have that:
\[|D\td D'|=|D|+|D'|\]
Thus 
\[\Psi_{\diamond}[D,D'](\z^\gamma)=D\circ D'(\z^\gamma) - D\td D'(\z^\gamma)\in \A_{|\gamma|+|D|+|D'|}\]
Then, using the defining equality \eqref{eq: Psi_diamond} it is easy to prove by induction on $n$ that for all $D_1,\ldots,D_n\in \B_\D$:
\begin{equation}\label{eq: homogeneity Psi}
    \Psi_{\diamond}[D_1\cdots D_n](\z^\gamma)\in\A_{|\gamma|+|D_1|+\ldots+|D_n|}
\end{equation}
where we denoted again $\A_{|\gamma|+|D_1|+\ldots+|D_n|}:=\{0\}$ if $|\gamma|+|D_1|+\ldots+|D_n|<0$.

\subsection{Structure group and recentering maps.}
In the context of the structure group in regularity structures, the following Proposition is the key for the construction the recentering maps $\Gamma$ in regularity structures as made in \cite[Section 5]{jacques2023post}. \\
We recall the representation morphism $\rhobar_{\btr}:(\sym(L),\star_{\btr})\to (\End(\A),\circ)$ given by \eqref{eq: representation deformed algebra derivation}. The following proposition is a version of \cite[Proposition 5.2.]{jacques2023post} for $(L,\btr)$, with $\btr:=\tr+\td$:

\begin{prop}\label{prop: finiteness representation} 
For all $\beta\in\M$ the set:
\begin{equation*}
    \left\{(u,\z^\gamma)\in \B_{\sym(L)}\times \B_\A,~\la \rhobar_{\btr}(u)(\z^\gamma),\z^\beta\ra\ne 0\right\},
\end{equation*}
is finite, where $B_{\sym(L)}$ is the monomial basis of $L$ given by \eqref{eq: monomial basis}, $\B_\A=\{\z^\gamma\}_{\gamma\in\M}$ is the monomial basis of $\A$ and $\la\cdot,\cdot\ra:\A^{\otimes 2}\to \A$ is the canonical pairing relative to $\B_\A$.
\end{prop}
\begin{proof}
For all $u=\prod_{i= 1}^d (\1\otimes \partial_i)^{m_i}(\z^{\gamma_1}\otimes D^{\n_1})\cdots (\z^{\gamma_\ell}\otimes D^{\n_\ell})\in \Bbar_{\sym(L)}$ where $(m_1,\ldots,m_d)\in \N^d$ and $(\gamma_i, \n_i)\in \M\times \N^d$ for all $i\in\{1,\ldots,\ell\}$, we have that:
\[\rhobar_{\btr}(u)(\z^\gamma)=\z^{\gamma_1+\ldots +\gamma_\ell}\cdot\Psi_\diamond\left[\prod_{i= 1}^d \partial_i^{m_i}\cdot D^{\n_1}\cdots D^{\n_\ell}\right](\z^\gamma)\]
Thus, if we fix $\beta\in\M$:
\[\la \rhobar_{\btr}(u)(\z^\gamma),\z^\beta\ra\ne 0 \Rightarrow \ell \alpha\leq |\gamma_1|+\ldots +|\gamma_\ell|\leq|\beta|\]
and since $\alpha >0$, we have the boundedness $\ell\leq \lfloor \frac{|\beta|}{\alpha}\rfloor$ and using that the set $\M_{\leq \kappa}:=\bigcup_{\kappa'\leq \kappa} \M_\kappa$ is finite, there is a finite choice of $\gamma_i\in\M_{\leq |\beta|}$ for all $i\in\{1,\ldots,\ell\}$ and hence by the condition $|\n_i|\leq |\gamma_i|$, there is also a finite possible choice for the $\n_i$'s. Finally, using \eqref{eq: homogeneity Psi}, we have that:
\[\Psi_\diamond\left[\prod_{i= 1}^d \partial_i^{m_i}\cdot D^{\n_1}\cdots D^{\n_\ell}\right](\z^\gamma)\in \A_{|\gamma|+\sum_{i=1}^d m_i - \n_1-\ldots-\n_\ell}\]
and again if $\la \rhobar_{\btr}(u)(\z^\gamma),\z^\beta\ra\ne 0$ we obtain the equality:
\[|\gamma|+\sum_{i=1}^d m_i - \n_1-\ldots-\n_\ell +|\gamma_1|+\ldots +|\gamma_\ell|=|\beta| \]
which shows that both $|\gamma|$ and $\sum_{i=1}^d m_i$ are bounded and hence there exists a finite choice for $\gamma\in\A_{\leq |\beta|+\sum_{k=1}^\ell (|\gamma_k|-|\n_k|)}$ and for $m_i\in \N$.
\end{proof}

As in \cite[Definition 2.19.]{jacques2023post}, $G$ is defined as the \textbf{character group} of the symmetric algebra $(\sym(L),\ast,\ind)$. That is to say, denoting $\sym(L)^*$ the dual space of linear maps $\sym(L)\to \R$, we have that $G$ is given as:
$$G:=\{f\in \sym(L)^*,~f(\ind)=1~\wedge~f(u \ast v)=f(u)f(v)~ \forall u,v\in L\}$$
We also recall from \cite[Proposition 2.20.]{jacques2023post} that the group structure on $G$ is given by: $(\star_{\btr} ,\ind^*)$, where the unit element is given by duality as $\mathds{1}^*(\cdot):=\la \mathds{1},\cdot\ra$, and the product is given by:
\begin{equation}\label{eq:convprod}
f_1\star_{\btr}  f_2  := m_\R(f_1\otimes f_2)\Delta_{\star_{\btr}}  .
\end{equation}
where $m_\R$ denotes the multiplication in $\R$.\\

We mention here that by the finiteness property of Proposition \ref{prop: finiteness representation}, the map:
\[\writefun{\rhobar_{\btr}}{\sym(L)\otimes \A}{\A}{u \bm{\otimes} \z^\beta}{\rhobar_{\btr}(u)(\z^\beta)}\]
can be dualised, via the canonical pairing $\la\cdot,\cdot\ra:\A^{\otimes 2}\to \A$ relative to $\B_\A$ into a left comodule map, the map $\rhobar_{\btr}^\ast:\A\to \sym(L)\otimes \A$ defined by:
\[\rhobar_{\btr}^\ast (\z^\gamma) :=\sum_{\beta\in\M}\sum_{u\in\B_{\sym(L)}}\la\rhobar(u)(\z^\beta),\z^\gamma\ra 
(u \otimesbold \z^\beta)\]
which is a coaction map, that is to say $(\A,\rhobar_{\btr}^\ast)$ is a left $(\sym(L),\Delta_{\star_{\btr}})$-comodule, that is to say:
\[(\id \otimesbold \rhobar_{\btr}^\ast)\rhobar_{\btr}^\ast =(\Delta_{\star_{\tr}}\otimesbold \id)\rhobar_{\btr}^\ast\]
This last equality is easily verifyed using the fact that for all $u_1,u_2\in $ equality \eqref{eq: representation algebra derivation2} as in \cite[Proposition 3.21]{jacques2023post}.\\

Similar to the construction in regularity structures \cite[equality 8.17]{hairer2013solving}, we can define recentering maps $\Gamma^{\btr}_f$ for all $f\in \sym(L)^*$ as:
\[\Gamma^{\btr}_f(\z^\gamma):=(f\otimesbold \id)\rhobar_{\btr}^\ast(\z^\gamma)\]
Then, observing that for all $f_1,f_2\in \sym(L)^\ast$:
\[\Gamma^{\btr}_{f_2}\circ \Gamma^{\btr}_{f_1}(\z^\gamma)=(f_1\otimesbold f_2\otimesbold \id)(\id \otimesbold \rhobar_{\btr}^\ast)\rhobar_{\btr}^\ast \z^\gamma\]
and that:
\[\Gamma^{\btr}_{f_1\star_{\btr}  f_2}(\z^\gamma)=(f_1\otimesbold f_2\otimesbold \id)(\Delta_{\star_{\tr}}\otimesbold \id)\rhobar_{\btr}^\ast \z^\gamma\]
where $f_1\otimesbold f_2$ is considered as being an element of $(\sym(L)\otimesbold\sym(L))^*$ which is defined for all $u_1,u_2\in \sym(L)$ as:
\[(f_1\otimesbold f_2)(u_1\otimesbold u_2)=f_1(u_1)f_2(u_2),\]
we conclude with the composition formula for recentering maps:
\[\Gamma^{\btr}_{f_1\star_{\btr}  f_2}=\Gamma^{\btr}_{f_2}\circ \Gamma^{\btr}_{f_1}\]

\medskip
\section{Appendix}\label{sec: Appendix}
\subsection{Proof of the Lemmas.}\label{subsec: Appendix proofs}
We start to write here below the proof of our algebraic first Bianchi identity \ref{lem: the first Bianchi identity}, although this proof has already been writen in \cite[Chapter III,$\S$6, Proposition 1]{nomizu1956lie} for the classical geometric setting where $(L,[\cdot,\cdot])$ is the Lie algebra of smooth sections of fiber bundles. However the proof of \cite{nomizu1956lie} does not requires any geometric considerations and can be transcribed as originally written. We write below the proof, for the seek of completnes, since we adopt a general algebraic point of view and since references about Bianchi identities in the general setting of non-vanishing torsion are difficult to find and quite old.

\begin{proof}[Proof of \ref{lem: the first Bianchi identity} ]
    Let $x,y,z\in L$, for simplicity, we denote here $\tors{}$ for $\tors{\td,[\cdot,\cdot]}$, $\curv{}$ for $\curv{\td,[\cdot,\cdot]}$.
    As given in \eqref{eq: torsion}, the torsion is defined by:
    $$\tors{}(x,y):=x\td y-y\td x-[x,y]$$
    We apply the operator $\tors{}(\cdot,z)=-\tors{}(z,\cdot)\in\End(L)$ to both sides of the last equality, which amounts to:
    $$\tors{}(\tors{}(x,y),z)=\tors{}(x\td y,z)+\tors{}(z,y\td x)+\tors{}(z,[x,y])$$
    Summing boths sides of the last equality over ciclic permutations of $(x,y,z)$, we obtain:
    \begin{align*}
        \mathfrak{S}(\tors{}(\tors{}(x,y),z))&=\mathfrak{S}\Big(\tors{}(x\td y, z) + \tors{}(z, y\td x) +\tors{}(z,[x, y])\Big)\\
        &=\mathfrak{S}(\tors{}(x\td y, z)) + \underbrace{\mathfrak{S}(\tors{}(z, y\td x))}_{\mathfrak{S}(\tors{}(y, x\td z))} + \underbrace{\mathfrak{S}(\tors{}(z,[x, y]))}_{\mathfrak{S}(\tors{}(x,[y, z]))}\\
        &=\mathfrak{S}\Big(\underbrace{\tors{}(x\td y, z)+\tors{}(y, x\td z)}_{x\td\tors{}(y,z)-(x\td \tors{})(y,z)}+\tors{}(x,[y, z])\Big)
    \end{align*}
    Then, using the definition of $\tors{}(\cdot,\cdot)=[\cdot,\cdot]_{\diamond}-[\cdot,\cdot]$, we compute:
    \begin{align*}
        x\td \tors{}(y, z)+&\tors{}(x,[y, z])\\
        &=x\td [y,z]_{\diamond}-x\td [y,z]+[x,[y, z]]_{\diamond}-[x,[y, z]]\\
        &=x\td [y,z]_{\diamond}-[y, z]\td x-[x,[y, z]]
    \end{align*}
    Then, using the linearity of $\mathfrak{S}$ and the fact that $[\cdot,\cdot]$ satisfyes the Jacobi identity $\mathfrak{S}([x,[y, z]])=0$, we obtain that:  
    
    \begin{align*}
        &\mathfrak{S}(\tors{}(\tors{}(x,y),z))\\
        &= \mathfrak{S}(x\td (y\td z)) -\underbrace{\mathfrak{S}(x\td (z\td y))}_{\mathfrak{S}(y\td (x\td z))} - \underbrace{\mathfrak{S}([y, z]\td x)}_{\mathfrak{S}([x,y]\td z)} -\mathfrak{S}((x\td \tors{})(y,z))\\
        &= \mathfrak{S}\Big(\underbrace{x\td (y\td z) - y\td (x\td z) - [x,y]\td z}_{\curv{}(x,y,z)}\Big) -\mathfrak{S}\Big((x\td \tors{})(y,z) \Big)
    \end{align*}

\end{proof}

\begin{proof}[Proof of Lemma \ref{lem: Der alg. ==> Der Lie alg.}] The proof is a simple computation where we first use the hypothesis $x\tr \cdot \in \Der(L,\td)$ and then we group the terms:
\begin{align*}
    x\tr [y, z]_{\diamond}&= x\tr(y\td z) - x\tr(z\td y)\\
    &= (x\tr y)\td z + y\td (x\tr z) - (x\tr z)\td y - z\td (x\tr y)\\
    &= \Big((x\tr y)\td z-z\td (x\tr y)\Big) + \Big(y\td (x\tr z)- (x\tr z)\td y\Big)\\
    &=[x\tr y,z]_{\diamond} + [y,x\tr z]_{\diamond}
\end{align*}
\end{proof}

\begin{proof}[Proof of Lemma \ref{lemma: constant torsion post-Lie identity}]
Since $(L,\tr,[\cdot,\cdot])$ is a post-Lie algebra we have that $X\tr \in \Der(L,[\cdot,\cdot])$ and since we have supposed that the equality \eqref{eq: compatibility derivation} is satisfied the Lemma \ref{lem: Der alg. ==> Der Lie alg.} indicates that also $X~\tr \in \Der(L,[\cdot,\cdot]_{\diamond})$, thus we have:
\begin{align*}
    x\tr\lb y,z\rb&=x\tr [y, z] - x\tr [y, z]_{\diamond}\\
    &=[x\tr y, z]+[y, x\tr z] - [x\tr y, z]_{\diamond} - [y, x\tr z]_{\diamond}
\end{align*}

Then, recombining the terms, and using the fact that $\td\tors{\td}=0$ is equivalent to
$$x\td\lb y, z \rb=\lb x\td y, z \rb + \lb y, x\td z \rb,$$ we finally get:
\begin{align*}
    x\tr\lb y,z\rb&= x\tr\lb y,z\rb + x\td\lb y,z\rb\\
    &=\underbrace{[x\tr y, z ]-[x\tr y, z]_{\diamond}}_{\lb x\tr y, z \rb}+ \underbrace{[y,x\tr z] - \lb y, x\tr z \rb_{\td}}_{\lb y,x\tr z \rb} + x\td\lb y,z\rb\\
    &=\underbrace{\lb x\tr y, z \rb+ \lb x\td y, z \rb}_{\lb x\tr y + x\td y, z \rb} + \underbrace{\lb y,x\tr z \rb + \lb y, x\td z \rb}_{\lb y,x\tr z + x\td z\rb}\\
    &=\lb x\tr y,z\rb+\lb y,x\tr z\rb
\end{align*}
\end{proof}

\begin{proof}Proof of Lemma \ref{lemma: curvature post-Lie identity}] We start expanding $\ass_{\btr}(x,y,z)$ the following way:
$$x\tr (y\btr z)= \underbrace{x\tr (y \tr z)}_{E_1(x,y,z)} + \underbrace{x\tr (y\td z)}_{E_2(x,y,z)} + \underbrace{x\td (y\tr z)}_{E_3(x,y,z)} + \underbrace{x\td (y\td z)}_{E_4(x,y,z)}$$
$$(x\tr y)\btr z = \underbrace{(x\tr y) \tr z}_{E_1'(x,y,z)} + \underbrace{(x\tr y)\td z}_{E_2'(x,y,z)} + \underbrace{(x\td y)\tr z}_{E_3'(x,y,z)} + \underbrace{(x\td y)\td z}_{E_4'(x,y,z)}$$
    Then:
    $$\ass_{\btr}(x,y,z)=(E_1-E_1' + E_2-E_2' + E_3-E_3' + E_4-E_4')(x,y,z)$$
    In the computation of $\ass_{\btr}(x,y,z)-\ass_{\btr}(y,x,z)$, we following arises:
    \begin{itemize}
        \item The terms in $E_2,E_2',E_3,E_3'$ vanish: indeed the compatibility condition \eqref{eq: compatibility derivation} that $y \tr \cdot~\in \Der(L,\td)$, which implies:
        \begin{equation*}
            E_3(x,y,z)-\Big(E_2-E_2'\Big)(y,x,z)=x\td (y\tr z) - y\tr (x\td z) + (y\tr x)\td z=0
        \end{equation*}
    Reversing $X$ and $Y$, the same compatibility condition \eqref{eq: compatibility derivation} is also:
    $$\Big(E_2-E_2'\Big)(x,y,z)-E_3(y,x,z)=0$$
        \item We also have by linearity:
    $$E_3'(y,x,z)-E_3'(x,y,z)=-[x, y]_{\diamond}\tr z$$
        \item Since $(L,\tr,[\cdot,\cdot])$ is supposed to be a post-Lie algebra, we have that:
        \[(E_1-E_1')(x,y,z)-(E_1-E_1')(y,x,z)=\ass_{\tr}(x,y,z)-\ass_{\tr}(y,x,z)=[x,y] \tr z\]
        \item By definition of the curvature tensor \eqref{eq: curvature}, we have that:
    \begin{align*}
        (E_4-E_4')(x,y,z)-(E_4-E_4')(x,y,z)&=\ass_{\td}(x,y,z)-\ass_{\td}(y,x,z)\\
        &=\curv{\td}(x,y,z)+[x,y] \td z-[x,y]_{\diamond}\td z
    \end{align*}
    \end{itemize}
    
    We finally, regrouping the terms and using linearity, we obtain the desired equality:
    \begin{align*}
        &\ass_{\btr}(x,y,z)-\ass_{\btr}(y,x,z)\\
        &=\underbrace{[x,y] \tr z + [x,y] \td z}_{[x,y] \btr z} - (\underbrace{[x,y]_{\diamond}\td z + [x,y]_{\diamond}\tr z}_{[x,y]_{\diamond}\btr z}) + \curv{\td}(x,y,z)\\
        &=([x,y]-[x,y]_{\diamond})\btr z + \curv{\td}(x,y,z)\\
        &=\lb x,y \rb\btr z + \curv{\td}(x,y,z)
        \end{align*}
    \end{proof}

\subsection{Post-Lie deformation conditions in coordinates.}\label{subsec: PL deformation in coordinates}
We consider as in Section \ref{sec: torsion and curvature} a Lie algebra $(L,[\cdot,\cdot])$ where $L$ is endowed with a basis $\{x_i\}_{i\in I}$ indexed by a set $I$. Given a bilinear operation $\td:L\otimes L\to L$ (the connection), it can be expressed in coordinates along the basis $\{x_i\}_{i\in I}$, as well as the Lie bracket $[\cdot,\cdot]$:
\begin{equation*}
    x_i\td x_j = \sum_{m\in I} \gamma_{i,j}^m x_m \qquad\text{and} \qquad [x_i,x_j] = \sum_{m\in I} \delta_{i,j}^m x_m
\end{equation*}
where the two sums are finite with real coefficients $\gamma_{i,j}^m,\delta_{i,j}^m\in \R$. \\

We recall the notion of torsion $\tors{\td,[\cdot,\cdot]}$ and curvature $\curv{\td,[\cdot,\cdot]}$ of $\td$ on $(L,[\cdot,\cdot])$ given by equalities \eqref{eq: torsion} and \eqref{eq: curvature} and we denote
\[\gamma_{[i,j]}^m:=\gamma_{i,j}^m-\gamma_{j,i}^m \qquad \text{and}\qquad \delta_{[i,j]}^m:=\delta_{i,j}^m-\delta_{j,i}^m\]
We get in coordinates the following decompositions, first for the torsion

$$\tors{\td,[\cdot,\cdot]}(x_i,x_j):=[x_i,x_j]_{\diamond}-[x_i,x_j]=\sum_{m\in I} (\gamma_{[i,j]}^m-\delta_{i,j}^m) x_m$$

then for the covariant derivative of the torsion
\begin{align*}
    (x_i\td&\tors{\td,[\cdot,\cdot]})[x_j,x_k]\\
    :=&x_i\td\left(\tors{\td,[\cdot,\cdot]}[x_j,x_k]\right) - \tors{\td,[\cdot,\cdot]}[x_i\td x_j,x_k] - \tors{\td,[\cdot,\cdot]}[x_j,x_i\td x_k]\\
    =&x_i\td [x_j,x_k]_{\diamond}-[x_i\td x_j,x_k]_{\diamond}-[x_j,x_i\td x_k]_{\diamond}\\
    &-x_i\td[x_j,x_k]_\circ+[x_i\td x_j,x_k]_\circ+[x_j,x_i\td x_k]_\circ\\
    =&\sum_{l,m}\Big(\gamma_{i,l}^m\gamma_{[j,k]}^l - \gamma_{i,j}^l\gamma_{[l,k]}^m  - \gamma_{[j,l]}^m\gamma_{i,k}^l -\gamma_{i,l}^m\delta_{j,k}^l  + \gamma_{i,j}^l\delta_{l,k}^m  + \delta_{j,l}^m\gamma_{i,k}^l\Big) x_m\\
    =&\sum_{l,m}\Bigg(\gamma_{i,l}^m\Big(\gamma_{[j,k]}^l-\delta_{j,k}^l\Big) - \gamma_{i,j}^l\Big(\gamma_{[l,k]}^m - \delta_{l,k}^m \Big)  - \Big(\gamma_{[j,l]}^m - \delta_{j,l}^m\Big)\gamma_{i,k}^l\Bigg) x_m
\end{align*}
and finally for the curvature tensor:
\begin{align*}
    \curv{\td,[\cdot,\cdot]}[x_i,x_j,x_k]&:=x_i\td (x_j\td x_k)-x_j\td (x_i\td x_k)-[x_i,x_j]_\circ\td x_k\\
    &=\sum_{l,m\in I}\Big(\gamma_{i,l}^m\gamma_{j,k}^l-\gamma_{j,l}^m\gamma_{i,k}^l-\delta_{i,j}^l\gamma_{l,k}^m\Big)x_m
\end{align*}

\begin{prop}\label{prop: conditions coordinates} Considering conditions of null torsion, null torsion derivative or null curvature, we obtain the following homogene polynomial equations in $\R[\{\gamma_{i,j}^k,\delta_{i,j}^k\}_{i,j,k\in I}]$:
    \begin{itemize}
        \item $\tors{\td,[\cdot,\cdot]}=0 \iff \forall (i,j,m)\in I^3$:
            \begin{equation}\label{eq: null torsion coordinates}
            \gamma_{[i,j]}^m=\delta_{i,j}^m
            \end{equation}
        \item $\td\tors{\td,[\cdot,\cdot]}=0 \iff \forall (i,j,k,m)\in I^4$:
            \begin{equation}\label{eq: torsion condition coordinates}
            \sum_{l\in I}\gamma_{i,l}^m\Big(\gamma_{[j,k]}^l-\delta_{j,k}^l\Big) - \gamma_{i,j}^l\Big(\gamma_{[l,k]}^m - \delta_{l,k}^m \Big)  - \Big(\gamma_{[j,l]}^m - \delta_{j,l}^m\Big)\gamma_{i,k}^l=0
            \end{equation}
        \item $\curv{\td,[\cdot,\cdot]}=0 \iff \forall (i,j,k,m)\in I^4$:
            \begin{equation}\label{eq: curvature condition coordinates}
            \sum_{l\in I}\gamma_{i,l}^m\gamma_{j,k}^l-\gamma_{j,l}^m\gamma_{i,k}^l-\delta_{i,j}^l\gamma_{l,k}^m=0
            \end{equation}
    \end{itemize}
\end{prop}

\begin{remark}
    After fixing a basis of $L$, all geometric post-Lie deformation are characterised by the solutions of the polynomial equations \ref{eq: null torsion coordinates}, \ref{eq: torsion condition coordinates} and \ref{eq: curvature condition coordinates} in $\R[\{\gamma_{i,j}^k,\delta_{i,j}^k\}_{i,j,k\in I}]$. We mention the Gröbner basis iterative method for solving polynomial equations, see \cite{buchberger1970algorithmisches} and \cite{aschenbrenner2008algorithm}.
\end{remark}

We consider now the context of Section \ref{sec: deformation post-Lie alg of derivations} where $L=\A\otimes \D$, with an associative and commutative algebra $(\A,\cdot)$ and a subspace of the space of derivations $\D\subset Der(\A)$ that is stable by $[\cdot,\cdot]_\circ$. Throughout this subsection, we fix a basis $\{D_i\}_{i\in I}$ indexed by a set $I$. Given a bilinear operation $\td:\D\otimes\D\to\D$, it can be expressed as before in coordinates along the basis $\{D_i\}_{i\in I}$, as well as the commutator bracket $[\cdot,\cdot]_\circ$:
\begin{equation*}
    D_i\td D_j = \sum_{m\in I} \gamma_{i,j}^m D_m \qquad\text{and} \qquad [D_i,D_j]_\circ = \sum_{m\in I} \delta_{i,j}^m D_m
\end{equation*}
where the two sums are finite with real coefficients $\gamma_{i,j}^m,\delta_{i,j}^m\in \R$. Note that by anti-commutativity $\delta_{i,j}^m=-\delta_{j,i}^m$. Then we can express the structure $(\btr,\lb\cdot,\cdot\rb)$ given by the equalities \eqref{eq: post-lie product geometric derivation alg} and \eqref{eq: post-lie bracket geometric derivation alg} along the basis, which gives:
\begin{align*}
    (a\otimes D_i)\btr (b\otimes D_j)&=a D_i(b)\otimes D_j + \sum_{l\in I} ab\otimes \gamma_{i,j}^l D_l\\
\lb a\otimes D_i,b\otimes D_j\rb&=\sum_{l\in I} ab\otimes (\delta_{i,j}^l-\gamma_{i,j}^l) D_l
\end{align*}

Combining Theorem \ref{th: post-Lie conditions on A tens Der(A)} and Proposition \ref{prop: conditions coordinates} we easily obtain:
\begin{prop}
If the coefficients $\gamma$ and $\delta$ satisfy \eqref{eq: null torsion coordinates} and \eqref{eq: curvature condition coordinates}, then $(\A\otimes \D,\btr)$ is a pre-Lie algebra.\\
If the coefficients $\gamma$ and $\delta$ satisfy \eqref{eq: torsion condition coordinates} and \eqref{eq: curvature condition coordinates}, then $(\A\otimes \D,\btr,\lb\cdot,\cdot\rb)$ is a post-Lie algebra.
\end{prop}

In regularity structures, if we consider the space $L_0$ defined by \eqref{eq: post-Lie algebra L_0}, the equations \ref{eq: null torsion coordinates}, \ref{eq: torsion condition coordinates} and \ref{eq: curvature condition coordinates} simplifies, indeed in \cite{LOT} and \cite{jacques2023post}, $\D$ is the space freely generated by the familly of derivations:
$$\B_\D:=\{\partial_i\}_{i\in\{1,\ldots,d\}}\cup \{D^{(\n)}\}_{\n\in\N^d}.$$
We have that for all $i\in\{1,\ldots,n\}$ and $\n=(n_1,\ldots,n_d)\in\N^d\setminus \{\0\}$:
    \begin{equation*}
        [D^{(\n)}, \partial_i]_\circ=n_iD^{(\n-\e_i)}.
    \end{equation*}
    Thus:
    $$\delta_{\n,i}^{\m}=\left\{
    \begin{array}{ll}
    n_i \qquad \text{if}~ \m=\n-\e_i,\\
    0\qquad \text{else}
    \end{array}
    \right.$$
    And since $[\partial_i,\partial_j]_\circ=[\partial_i,D^{(\0)}]_\circ=[D^{(\n)},D^{(\m)}]_\circ=0$ we have that:
    \[\delta_{i,j}^k=\delta_{i,j}^{\n}=\delta_{\p,\q}^\n=\delta_{\p,\q}^i=\delta_{\0,i}^{j}=\delta_{\0,i}^{\n}=0\]

\medskip

\subsection{Link between geometric post-Lie deformations of the derivation post-Lie algebra and the Munthe-Kaas--Lundervold post-Lie algebra}\label{subsec: appendix M-K--L post-Lie alg}
Assuming that $\A$ is an unitary algebra, let us denote by $\1$ the unit element of $\A$. We consider the post-Lie algebra structure $(\A\otimes \Der(\A),\btr,\lb\cdot,\cdot\rb)$ given in Theorem \ref{th: post-Lie conditions on A tens Der(A)}, which is the $\td$-gpL deformation of the canonical post-Lie algebra $(\A\otimes \Der(\A),\tr,[\cdot,\cdot])$ in \cite{jacques2023post} defined in Theorem \ref{theo: post-Lie structure from derivations}. We have a canonical injective morphism of post-Lie algebras:
$$\writefun{\eta}{(\Der(\A),\td,-\tors{\td,[\cdot,\cdot]_\circ})}{(\A\otimes \Der(\A),\btr,\lb\cdot,\cdot\rb)}{X}{\1 \otimes X}$$
The morphism property is a simple fact that the unit element $\1$ satifies that $D(\1)=0$ for all $D\in\Der(\A)$, hence we have:
\begin{align*}
    \eta(D_1)\btr\eta(D_2)&= (\1 \otimes D_1)\btr (\1 \otimes D_2)\\
    &=(\1 \otimes D_1)\tr(\1 \otimes D_2)+(\1 \otimes D_1)\td(\1 \otimes D_2)\\
    &=\underbrace{D_1(\1)\otimes D_2}_{=0} + \underbrace{\1\otimes D_1 \td D_2}_{=\eta(D_1 \td D_2)}
\end{align*}

\begin{align*}
    \eta(-\tors{\td,[\cdot,\cdot]_\circ}(D_1,D_2))&=\eta([D_1,D_2]_\circ-[D_1,D_2]_{\diamond})\\
    &=\1\otimes [D_1,D_2]_\circ - \1\otimes [D_1,D_2]_{\diamond}\\
    &=[\1\otimes D_1,\1\otimes D_2] - [\1\otimes D_1,\1\otimes D_2]_{\diamond}\\
    &=\lb \underbrace{\1\otimes D_1}_{=\eta(D_1)},\underbrace{\1\otimes D_2}_{\eta(=D_2)}\rb
\end{align*}

Now, if we consider a smooth differential manifold $\M$ and we specify $\A:=\mathcal C^\infty(\M)$ endowed with the pointwise product and the unit element $\1$ is given by the constant function equal to $1$: $\1(x)=1$ for all $x\in \M$, and we consider a connection $\nabla$ which has constant torsion and null curvature, then we can consider the post-Lie algebra $(\mathcal C^\infty(\M)\otimes \mathfrak X(\M),\btr,\lb\cdot,\cdot\rb)$ where the product $\td$ is given by the connection: 
\[X\td Y:= \nabla_X Y\]
Then if we consider the Munthe-Kaas-Lundervold post-Lie algebra\\ $(\mathfrak X(\M),\nabla,-\tors{\nabla,[\cdot,\cdot]_J})$, where $[\cdot,\cdot]_J$ denote the Jacobi-Lie bracket of vector fields, we have a canonical injective morphism of post-Lie algebra, which is given by:
\[\writefun{\eta}{(\mathfrak X(\M),\nabla,-\tors{\nabla,[\cdot,\cdot]_J})}{(\mathcal C^\infty(\M)\otimes \mathfrak X(\M),\btr,\lb\cdot,\cdot\rb)}{X}{\1 \otimes X}\]

For the opposite way, it is easy to prove that the representation  map $\rho:\A\otimes \Der(\A)\to \Der(\A)$ given as before by:
\[\rho(a\otimes D)=a\cdot D,\]
is in that present context, a surjective post-Lie algebra morphism:
$$\writefun{\rho}{(\mathcal C^\infty(\M)\otimes \mathfrak X(\M),\btr,\lb\cdot,\cdot\rb)}{(\mathfrak X(\M),\nabla,-\tors{\nabla,[\cdot\cdot]_J})}{f\otimes X}{f\cdot X}$$
Indeed:
\begin{align*}
    \rho\Big((f_1\otimes X_1)\btr(f_2\otimes X_2)\Big)&=f_1X_1(f_2)X_2 + f_1f_2\nabla_{X_1}X_2\\
    &=f_1X_1(f_2)X_2 + f_2\nabla_{f_1X_1}X_2\\
    &=\nabla_{f_1X_1} (f_2 X_2)\\
    &=\nabla_{\rho(f_1\otimes X_1)} \Big(\rho(f_2\otimes X_2)\Big)
\end{align*}
And in that case, $\eta$ is a section of $\rho$ in the sense that:
$$\rho\circ \eta = \ind_{\mathfrak X(\M)}$$
\medskip

\printbibliography
\end{document}